\documentclass[11pt,a4paper]{article} 

\usepackage[utf8]{inputenc} 
\usepackage[T1]{fontenc}    

\usepackage[margin=1in]{geometry} 
\usepackage{microtype}            

\usepackage{amsmath,amssymb,amsthm,mathrsfs}
\usepackage{mathtools} 
\usepackage{graphicx}             
\usepackage{booktabs}             
\usepackage{enumitem}             
\usepackage{dsfont}
\usepackage{hyperref}

\newtheorem{theorem}{Theorem}[section]
\newtheorem{lemma}[theorem]{Lemma}
\newtheorem{proposition}[theorem]{Proposition}
\newtheorem{corollary}[theorem]{Corollary}

\theoremstyle{definition}
\newtheorem{definition}[theorem]{Definition}

\theoremstyle{remark}
\newtheorem{remark}[theorem]{Remark}
\newcommand{\R}{\mathbb{R}}
\newcommand{\E}{\mathbb{E}}
\newcommand{\Z}{\mathbb{Z}}
\newcommand{\Prob}{\mathbb{P}}
\newcommand{\conv}{\mathrm{conv}}

\newcommand{\ls}{\le}
\newcommand{\gr}{\ge}

\newcommand{\vol}{{\rm{vol}}}

\begin{document}

\title{Discrete log-concavity and threshold phenomena for atomic measures}

\author{
  Silouanos Brazitikos\thanks{University of Crete, Greece. Email: \texttt{silouanb@uoc.gr}}
  \and
  Minas Pafis\thanks{National and Kapodistrian University of Athens, Greece. Email: \texttt{mipafis@math.uoa.gr}}
}

\date{} 

\maketitle

\begin{abstract}
We investigate threshold phenomena for random polytopes $K_N=\conv\{X_1,\dots,X_N\}$ generated by i.i.d.\ samples
from an atomic law $\mu$. We identify and provide a missing justification in the
discrete-hypercube threshold argument
of Dyer--F\"uredi--McDiarmid, where the supporting half-space estimate
is derived via a smooth (gradient/uniqueness) step
that can fail at boundary contact points. We then compare threshold-driving mechanisms in the continuous
log-concave setting -- through the Cram\'{e}r transform and Tukey's half-space depth -- with their discrete analogues. Within this framework,  we establish a sharp threshold for
lattice $p$-balls $\mathbb{Z}^n \cap rB_p^n$. Finally, we present structural counterexamples
showing that sharp thresholds need not hold in general discrete log-concave settings.
\end{abstract}

\section{Introduction}

Let $\mu$ be a Borel probability measure on $\R^n$ and $X_1,X_2,\dots$ be i.i.d. random vectors in $\mathbb{R}^n$ with law $\mu$. For $N\ge n+1$ set
\[
K_N:=\conv\{X_1,\dots,X_N\}\subseteq\R^n.
\]
There are two natural expectations associated with this model.

\paragraph{}
For general $\mu$ one considers the \emph{expected captured mass}
\[
F_{n,N}(\mu)\ :=\ \E_{\mu^N}\,\big[\mu(K_N)\big]
\ =\ \int_{\R^n}\Prob\left(x\in K_N\right)\,d\mu(x),
\]
a quantity that has attracted recent attention; see e.g. \cite{BC,BGP2,BGP,Giannop-Tziotziou,Paf}.
When $\mu$ is supported on a convex body $K\subseteq\R^n$, or on its boundary, one may also study the normalized expected volume
\[
V_{n,N}(\mu,K)\ :=\ \frac{1}{\vol_n(K)}\,\E_{\mu^N}\big[\vol_n(K_N)\big],
\]
which is classical in stochastic convex geometry; see \cite{BCGTT,BKT,CTV,DFM,FPT,GaGia,Pivov}. In the case where $\mu=\mu_K$ is the uniform probability measure on $K$, the two quantities coincide, i.e. $F_{n,N}(\mu_K)=V_{n,N}(\mu_K,K)$. The present work focuses primarily on the functional $F_{n,N}(\mu)$, which remains meaningful for atomic measures.

\paragraph{Thresholds and the half-space mechanism.}
A basic question is whether $F_{n,N}(\mu)$ (or $V_{n,N}(\mu,K)$) exhibits a sharp transition
from near $0$ to near $1$ as $N$ grows, typically at an exponential scale $e^{cn}$ (see for example \cite{BC,FPT,GaGia,Paf}),
although in certain regimes the transition has been shown to occur at
super-exponential scales in $n$ (see e.g.  \cite{BCGTT,BC}). A convenient way to encode the geometry is via half-spaces. Define the (population) half-space depth
\[
q_\mu(x):=\inf_{\theta\in S^{n-1}} \mu\big(\{y:\ \langle y,\theta\rangle\ge \langle x,\theta\rangle\}\big).
\]
Depth enters because $x\notin K_N$ if and only if there exists a closed half-space $H^+$ containing $x$ whose interior avoids all samples, and hence
estimates for $q_\mu$ translate into upper/lower bounds on $\E_{\mu^N} \big[\mu(K_N)\big]$ by standard union-bound and separation arguments (in various forms,
already present in \cite{DFM} and developed further in later works).

\paragraph{The rate-function viewpoint.}
Large-deviation heuristics suggest choosing as comparison sets the sublevel sets of a convex rate function. With
\[
\Lambda_\mu(\xi):=\log\int_{\R^n} e^{\langle \xi,z\rangle}\,d\mu(z),
\qquad
\Lambda_\mu^\ast(x):=\sup_{\xi\in\R^n}\{\langle \xi,x\rangle-\Lambda_\mu(\xi)\},
\]
one sets
\[
B_r(\mu):=\{x\in\R^n:\ \Lambda_\mu^\ast(x)\le r\}.
\]
The threshold method then aims to prove supporting half-space estimates of the schematic form
\[
\inf_{x\in B_r(\mu)} q_\mu(x)\ \gtrsim\ e^{-r},
\]
which, combined with the depth-to-hull inequalities, yield a threshold description in terms of a typical scale of $\Lambda_\mu^\ast$.
In the continuous log-concave setting, this paradigm has been made precise in recent work, and sharp depth--Cram\'er comparisons
have been established (see, e.g., \cite{BC, BGP}).

\paragraph{The issue in the discrete cube proof of DFM.}
For atomic measures, the supporting half-space step contains a genuine subtlety: For a translation argument to work, $B_{r}(\mu)$ must be defined appropriately so as to be compact. However, in this case, a supporting
hyperplane to $B_r(\mu)$ at a boundary point $x$ need not be determined by a unique normal direction, and $\Lambda_\mu^\ast$ need not be
differentiable at $x$. In such cases, the frequently used ``smooth'' representation
\[
H^+(x)=\{y:\ \langle \nabla \Lambda_\mu^\ast(x),y-x\rangle\ge 0\}
\]
is not available: the correct object is the \emph{normal cone} (or subdifferential) at $x$, which may have dimension larger than $1$.

Our point is that this phenomenon is not merely a complication of exotic examples: it can occur in the hypercube setting, because the natural
convex sets governing the argument may touch boundary faces where the Legendre transform is not differentiable and where supporting hyperplanes
are not singled out by a gradient. In the discrete-hypercube threshold argument of \cite{DFM}, the lower-threshold step implicitly appeals to a
smooth/uniqueness principle for selecting and controlling supporting half-spaces; however, in the genuinely atomic regime this principle requires
additional justification. As stated, the argument does not address this boundary-contact scenario, leaving a gap precisely at the point where one
must convert geometric support information into a uniform lower bound on half-space mass.

\paragraph{Two directions in this paper (one positive, one negative).}
Our contributions split into two parts, with different scope and different conclusions.

\smallskip

(I) \emph{A corrected positive theory under independence.}
Assume $X=(X_1,\dots,X_n)$ has i.i.d. symmetric coordinates with atomic distribution.
We work within this setting to recover a threshold picture, relying on the method that first appeared in \cite{DFM} and was subsequently developed in \cite{GaGia,Paf}. We provide the missing justification of the Dyer--F\"uredi--McDiarmid method.
As a consequence, we obtain a threshold phenomenon for probability measures whose i.i.d.\ coordinates are symmetric, compactly supported, with mass at the endpoint.
We also give a p.m.f related condition that guarantees threshold phenomena for unbounded discrete log-concave marginals.  
In this way, a large class of discrete log-concave product measures falls back into the reach of the threshold theory—either because they are covered by the corrected Dyer–F\"uredi–McDiarmid argument once the technical gap is repaired, or because their mass functions satisfy the condition that guarantees the threshold mechanism.

\smallskip

(II) \emph{A structural framework and counterexamples beyond independence.}
We also develop a discrete convex-analytic framework for general discrete log-concave laws via convex-extensible potentials and log-concave
extensions to $\R^n$. This part is \emph{not} a general threshold theorem. Its purpose is to clarify what can go wrong in the discrete world:
we give integrability criteria for the extension of the p.m.f, and we construct explicit families showing that sharp thresholds may fail even under natural
discrete log-concavity assumptions. In particular, we construct probability measures on $\mathbb{R}^n, \, n \geq 2$, having specific properties, but with $\E_\mu[\Lambda_\mu^\ast]=+\infty$. The same properties for one-dimensional probability measures, would imply finiteness of all moments of $\Lambda_\mu^\ast$. We also provide examples where the variance-type
parameter governing sharp thresholds in the continuous log-concave theory diverges, thereby preventing a Brazitikos--Chasapis / Brazitikos--Giannopoulos--Pafis
type threshold conclusion \cite{BC,BGP}. The examples in the last section should therefore be read as \emph{measures with no  threshold},
and as evidence that substantial structure (such as independence together with quantitative regularity) is essential in atomic settings.

\paragraph{Examples and calibration.}
To anchor the discussion, we include benchmark computations and contrasting regimes. For instance, for the uniform measure on $\{0,1\}^n$, $\mu(K_N)$ equals the fraction of distinct sampled vertices, leading to an exact formula and a coupon-collector scale.
We also analyze the uniform measure on lattice points of $rB_p^n$, where $F_{n,N}(\mu)$ undergoes a transition on sub-exponential scales, and
we explain how these models fit into (or fall outside) the threshold theory in \cite{BC,BGP}.

The next section sets notation and states our main results.

\section{Notation and main statements}
First, we introduce some basic notation and definitions. We work in ${\mathbb R}^n$, which is equipped with the standard inner product  $\langle\cdot ,\cdot\rangle $. We denote the corresponding Euclidean norm by $\|\cdot \|_2$, and write $B_2^n$ for the Euclidean unit ball and $S^{n-1}$ for the unit sphere. Volume in $\mathbb{R}^n$ is denoted by $\vol_n$. For a set $A \subseteq \mathbb{R}^n$ we write $|A|$ for its cardinality and $\mathds{1}_A$ for its indicator function. For $p \geq 1$ we denote by $\|\cdot\|_p$ the $\ell_p$-norm and by $B_p^n$ the corresponding unit ball. We also define $\|x\|_0=|\{1\leq i \leq n: x_i \neq 0\}|$, i.e. the number of non-zero coordinates of some $x \in \mathbb{R}^n$.

For any hyperplane $$H=\{y \in \mathbb{R}^n: \langle y,u\rangle =\langle x , u \rangle\}$$ that passes through $x \in \mathbb{R}^n$ and is perpendicular to some $u \in S^{n-1}$ we denote by $H^+$ the half-space $$H^+=\{y \in \mathbb{R}^n: \langle y,u\rangle \geq \langle x , u \rangle\}$$

For a function $g: \mathbb{R}^n \to \mathbb{R \cup \{+\infty\}}$, not identically infinite, we define its convex conjugate $$g^\ast(x)\coloneq\sup_{\xi \in \mathbb{R}^n} \ \{  \langle x,\xi\rangle - g(\xi)\}, \ x \in \mathbb{R}^n.$$ The biconjugate of $g$ satisfies $$g^{\ast \ast}(x)\coloneq(g^\ast)^\ast(x)=\sup \{ c(x)| \ c: \mathbb{R}^n \to \mathbb{R} \cup \{+\infty\} \text{ is convex, lsc and } c\leq g  \}.$$ We refer to \cite{Rock} for basic facts from convex analysis.

We use $f(x)\sim g(x)$, as $x \to A$, to declare that  $\lim\limits_{x \to A} \frac{f(x)}{g(x)}=1$. If $f,g \geq 0$, $f=O(g(n))$ means that $f(n) \leq Cg(n)$ for an absolute constant $C>0$. We also write $f(n)=o(g(n))$ for $\lim\limits_{n \to \infty} \frac{f(n)}{g(n)}=0$ and $f(n)=\omega(g(n))$ for $\lim\limits_{n \to \infty} \frac{f(n)}{g(n)}=+\infty$.

We say that a Borel probability measure $\mu $ on $\mathbb{R}^n$ is even if $\mu (-B)=\mu (B)$ for every Borel subset $B$
of ${\mathbb R}^n$. A random vector $Y$ in $\mathbb{R}^n$ is symmetric, if its law is even.  

A Borel probability measure $\mu$ on $\mathbb{R}^n$ is called full-dimensional if $\mu(H)<1$ for every
hyperplane $H$ in ${\mathbb R}^n$. We also say that $\mu$ is log-concave if it is full-dimensional and
$$\mu(\lambda A+(1-\lambda)B) \gr \mu(A)^{\lambda}\mu(B)^{1-\lambda}$$
for any pair of compact sets $A,B$ in ${\mathbb R}^n$ and any $\lambda \in (0,1)$. Borell \cite{Borell-1974} has proved that, under these assumptions, $\mu $ has a log-concave density $f_{{\mu }}$. Recall that a function $f:\mathbb R^n \rightarrow [0,\infty)$ is called log-concave if its support $\{f>0\}$ is a convex set in ${\mathbb R}^n$ and the restriction of $\log{f}$ to it is concave. If $f$ is an integrable log-concave function on $\mathbb{R}^n$, there exist constants $A,B>0$ such that $f(x)\ls Ae^{-B\|x\|_2}$ for all $x\in {\mathbb R}^n$ (see \cite[Lemma~2.2.1]{BGVV-book}).

Let $\mu$ be a Borel probability measure on $\mathbb R^n$. The population or Tukey's half-space depth is the function $$q_\mu(x):=\inf_{\theta\in S^{n-1}} \mu\big(\{y:\ \langle y,\theta\rangle\ge \langle x,\theta\rangle\}\big)$$ and was introduced by Tukey in \cite{Tukey-1975}.  The moment generating function (m.g.f) of $\mu$ is defined by
\begin{equation*}\phi_{\mu }(\xi )=\int_{{\mathbb R}^n}e^{\langle\xi
,z\rangle }d\mu(z)\end{equation*}
and the log-m.g.f is the function $$\Lambda_{\mu}(\xi):=\log \phi_{\mu}(\xi).$$
The Cram\'{e}r transform $\Lambda_{\mu}^{\ast }$ of $\mu $ is the convex conjugate of $\Lambda_{\mu}$, i.e.
\begin{equation*}\Lambda_{\mu }^{\ast }(x):= \sup_{\xi\in {\mathbb R}^n} \left\{ \langle x, \xi\rangle - \Lambda_{\mu }(\xi )\right\}.\end{equation*}
Note that $\Lambda_{\mu}^{\ast}$ is a non-negative convex function.

For any $r>0$ we define
$$B_r(\mu)=\{x\in\mathbb{R}^n:\Lambda_{\mu}^{\ast}(x)\ls r\}.$$
For every $x \in \mathbb{R}^n$ we have that $q_{\mu } (x)\ls\exp (-\Lambda_{\mu }^{\ast }(x))$ which is a direct
consequence of the definitions (see e.g. \cite[Lemma~3.1]{BGP}). We note that if $X$ is a random vector distributed according to $\mu$, then we also write $\phi_X, q_X, \, \Lambda_X,\, \Lambda_X^\ast $, respectively.

If the coordinates of $X$ are independent random variables, then $$\Lambda_X^\ast(x_1,\ldots,x_n)=\sum_{i=1}^n \Lambda_{X_i}^{\ast}(x_i).$$

For a probability measure $\mu$ on $\mathbb{R}^n$ we also define the quantity $$\beta(\mu):= \frac{{\rm Var}_{\mu} (\Lambda_{\mu}^\ast) }{(\E_{\mu}[\Lambda_{\mu}^\ast])^2},$$ that plays a crucial role in the threshold phenomena. Notice that if $\mu=\nu \otimes\cdots \otimes \nu \,$ ($k$-times) is a product measure, then $$\beta(\mu)=\frac{\beta(\nu)}{k}.$$ We also note that if $\mu$ is a probability measure on $\mathbb{R}$, the inequality $e^{\Lambda_{\mu}^\ast} \leq 1/q_\mu$ and \cite[Proposition 3.2 b)]{BC} yield  $\beta(\mu)<+\infty$. However, in $\mathbb{R}^n$ this quantity can be infinite. Giannopoulos and Tziotziou showed in \cite{Giannop-Tziotziou} that if $\mu$ is a log-concave probability measure on $\mathbb{R}^n$, then $\Lambda_{\mu}^\ast$ has all its moments finite. In our work, we associate the integrability of the continuous log-concave extension of a discrete log-concave p.m.f (see Definitions~\ref{def:convex-extension} and ~\ref{def:discrete-log-vectors}) with the finiteness of moments of $\Lambda_{\mu}^\ast$.

\begin{theorem}\label{thm:lambda-star-moments} Let $X$ be a discrete log-concave random vector in $\mathbb{Z}^n$ with p.m.f $p:\mathbb{Z}^n \to [0,1) $ and let $f$ be its log-concave extension. If $f$ is integrable, then for every $q \geq 1$ we have $$\mathbb{E}_X[(\Lambda_{X}^{\ast})^q] <+\infty.$$
\end{theorem}

Before we proceed to the threshold phenomena, we need to formalize this concept. Given a Borel probability measure $\mu$ on $\mathbb{R}^n$ we set 
\[
\varrho_1(\mu,\delta):= \sup\{r>0: \E_{\mu^N}[\mu(K_N)]\ls \delta, \hbox{ for every } N\ls \exp(r)\},
\]
and
\[
\varrho_2(\mu,\delta):= \inf\{r>0: \E_{\mu^N}[\mu(K_N)]\gr 1-\delta, \hbox{ for every } N\gr \exp(r)\},
\]
where $\mu^N=\mu \otimes\cdots\otimes\mu \, $ ($N$-times).

With this notation, we give the following definition.
\begin{definition}\label{def:sharp-threshold}
Let $(\mu_n)_{n\in\mathbb{N}}$ be a sequence of probability measures $\mu_n$ on $\mathbb{R}^n$, and $(T_n)_{n \in \mathbb{N}}$ a sequence of positive numbers. 

 We say that $(\mu_n)_{n\in\mathbb{N}}$ exhibits a weak threshold around $(T_n)_{n \in \mathbb{N}}$ if for every $\delta, \varepsilon \in (0,1)$ there exists $n_0(\delta,\varepsilon) \in \mathbb{N}$ such that for every $n \geq n_0$ \[
\varrho_1(\mu_n,\delta)\gr(1-\varepsilon)T_n \qquad\hbox{ and }\qquad \varrho_2(\mu_n,\delta)\ls(1+\varepsilon)T_n.
\]
We say that $(\mu_n)_{n\in\mathbb{N}}$ exhibits a sharp threshold around $(T_n)_{n \in \mathbb{N}}$ if for every $\varepsilon\in(0,1)$ there exists  $n_0=n_0(\varepsilon)\in\mathbb{N}$ such that for every $n\gr n_0$,
\[
\varrho_1(\mu_n,\delta_n)\gr(1-\varepsilon)T_n \qquad\hbox{ and }\qquad \varrho_2(\mu_n,\delta_n)\ls(1+\varepsilon)T_n,
\]
for some sequence $(\delta_n)_{n \in \mathbb{N}}$ with $\lim_{n\to\infty}\delta_n=0$.
\end{definition}

The theorem below builds on existing methods for threshold phenomena (e.g. \cite{DFM,GaGia,Paf}), either by repairing points in the original arguments or by showing that additional classes of measures naturally fall within the scope of the existing theory.

\begin{theorem}
    Let $\mu$ be an even  Borel  probability measure, and let $$x^\ast=x^\ast(\mu)=\sup \{ x \in \mathbb{R}: \mu([x,+\infty))>0\} \in (0,+\infty]$$ be its right endpoint. For every $n \in \mathbb{N}$ we set $\mu_n:=\mu\otimes\cdots\otimes\mu$ ($n$-times). If 
    \begin{enumerate}
       \item[{\rm (i)}] $\mu$ is compactly supported and $\mu(\{x^\ast\})>0$, or
       \item[{\rm (ii)}] $\mu$ is unbounded, discrete log-concave with p.m.f $(p(k))_{k \in \mathbb{Z}}=(e^{-g(k)})_{k \in \mathbb{Z}}$ that satisfies\\ $g(k+1)/g(k) \to 1$, as $ k \to \infty$,
\end{enumerate}
then $(\mu_n)_{n \in \mathbb{N}}$ exhibits a weak threshold, around $({\mathbb E}_{\mu_n}(\Lambda_{\mu_n}^{\ast}))_{n \in \mathbb{N}}$.
\end{theorem}

We further establish a sharp threshold using arguments inspired by the coupon-collector problem, and we locate the transition at a scale strictly smaller (in fact, polynomial) than the one in the known threshold results, which was  at least exponential.

\begin{theorem}\label{thm:lattice-ball-threshold}
Fix $r,p \geq 1$ and let $\mu_{n,r,p}$ be the uniform probability measure on $\mathbb{Z}^n \cap rB_p^n$. Then $(\mu_{n,r,p})_{n \in \mathbb{N}}$ exhibits a sharp threshold around some $(T_{n,r,p})_{n \in \mathbb{N}}$, with $T_{n,r,p} \sim \lfloor r^p \rfloor \log n$ as $n \to \infty$.
    \end{theorem}

Finally, we construct a sequence of discrete log-concave product measures that do not exhibit any threshold, in sharp contrast to the continuous log-concave setting.

\begin{theorem}
There exists a sequence $(\lambda_n)_{n \in \mathbb{N}}$ of discrete log-concave probability measures on $\mathbb{R}$ such that if $\mu_n= \lambda_1 \otimes\cdots\otimes \lambda_n$ for every $n \in \mathbb{N}$, then  $(\mu_n)_{n \in \mathbb{N}}$ does not exhibit any threshold.
\end{theorem}

\section{Discrete log-concave random vectors}

A recurring quantity in the threshold method is the mean (and more generally moments) of the convex
conjugate $\Lambda_X^\ast$ of the log-moment generating function $\Lambda_X$.
In particular, several threshold statements depend sensitively on whether $\E_X[\Lambda_X^\ast]$
is finite. Its finiteness is not automatic even in the continuous case, where it was verified first in \cite{BGP} for the uniform measure on a convex body and later for all log-concave measures with continuous density in \cite{Giannop-Tziotziou}.
For discrete models the question is naturally
linked to the possibility of passing from a discrete log-concave mass function on $\Z^n$
to a genuinely integrable log-concave function on $\R^n$.

\smallskip

This motivates the following more fundamental question:
\emph{given a discrete log-concave p.m.f $p$ on $\Z^n$ {\rm (}so that
$g=-\log p$ is convex-extensible in the sense of Murota in {\rm \cite{Mur})}, does its log-concave extension
$f=e^{-\tilde g}$ necessarily satisfy $\int_{\R^n} f<\infty$?}
If the answer is positive, then the analytic toolkit available for integrable log-concave functions
can be imported into the discrete setting, and in particular one obtains finiteness of all moments
of $\Lambda_X^\ast$ (see Theorem~\ref{thm:lambda-star-moments}).

\smallskip

The section is therefore organized around two parallel themes: the \emph{structural} notion of
discrete log-concavity (via convex-extensibility) and \emph{integrability criteria} for the
corresponding extension.
The proofs combine tools from convex analysis (biconjugation and recession cones),
elementary comparisons for log-concave functions, and lattice/Diophantine input used to rule out
the presence of ``too many'' lattice points inside large superlevel sets of $f$.
Concretely, we will use:
(i) separability of the convex extension in the product case (Lemma~\ref{lem:separability}),
(ii) boundedness of suitable superlevel sets via recession-cone arguments and the fact that a p.m.f
cannot charge infinitely many interior lattice points (Lemmas~\ref{lem:dioph}, \ref{lem:infinite-lattice}),
and (iii) in dimension two, a Kronecker-type density argument (Lemmas~\ref{lem:kronecker}, \ref{lem:R^2})
to obtain integrability without additional geometric assumptions on the support.
 
We begin with the standard definition of discrete log-concavity on $\mathbb{Z}$.

\begin{definition}[Discrete log-concavity on $\mathbb{Z}$]\label{def:one-dimension}
A sequence $(p_k)_{k\in\mathbb Z}$ of non-negative numbers is called \emph{discrete log-concave} if its support is a discrete interval and for all $k$ one has
\[ p_k^2\ge p_{k-1}p_{k+1}.\]
Equivalently,  the sequence of ratios $p_{k+1}/p_k$ is non-increasing on the support.

An integer valued random variable $Y$ is called discrete log-concave if its probability mass function (p.m.f) is a discrete log-concave sequence.
\end{definition}

In order to extend the notion of discrete log-concavity to $\mathbb{Z}^n$, we need the following definitions:

\begin{definition}[convex-extensible]\label{def:convex-extension}
    For some $g: \mathbb{Z}^n \to [0,+\infty]  $ which is not identically equal to $+\infty$ we define $\tilde{g}: \mathbb{R}^n \to [0,+\infty]$ by \begin{align*}
        \tilde{g}(x) &=\sup \{ a(x)| \ a:\mathbb{R}^n \to \mathbb{R} \text{  is affine and } a(k) \leq g(k) \text{ for every } k \in \mathbb{Z}^n \} 
        \\
        &=\sup \{ c(x)| \ c: \mathbb{R}^n \to \mathbb{R} \cup \{+\infty\} \text{ is convex, lsc and } c(k)\leq g(k) \text{ for every } k \in \mathbb{Z}^n \}, \ x \in \mathbb{R}^n.
    \end{align*}
    We say that $g$ is \textit{convex-extensible} if $\tilde{g}(k)=g(k)$ for every $k \in \mathbb{Z}^n$.
\end{definition}

Notice that for every such $g: \mathbb{Z}^n \to [0,+\infty]$, $\tilde{g}$ is convex, lower-semicontinuous and non-negative, since $a\equiv 0$ is affine and satisfies $a(k) \leq g(k)$ for every $k \in \mathbb{Z}^n$. Moreover, if ${\rm dom}(g)=\{k : g(k) < +\infty\}$ and ${\rm dom}(\tilde{g})=\{x:\tilde g(x) < +\infty\}$, then $$\conv({\rm dom}(g)) \subseteq {\rm dom}(\tilde{g}) \subseteq \overline{\conv}({\rm dom}(g)).$$ We will also need the following lemma.

\begin{lemma}\label{lem:separability}
    Let $g_1,\ldots,g_n : \mathbb{Z} \to [0,+\infty]$ which are not everywhere infinite and define $g: \mathbb{Z}^n \to [0,+\infty]$ by $g(k_1,\ldots,k_n)= \sum\limits_{i=1}^{n}g_i(k_i)$ for every $k \in \mathbb{Z}^n$. Then we have that $\tilde{g}(x_1,\ldots,x_n)=\sum\limits_{i=1}^{n}\tilde{g_i}(x_i)$ for every $x \in \mathbb{R}^n$.
\end{lemma}

\begin{proof}
    For every $1 \leq i \leq n$ we set $\hat{g_i}: \mathbb{R} \to [0,+\infty]$ with  $$ \hat{g_i}(x) =
    \begin{cases}
    g_i(x) & \text{, } x \in \mathbb{Z} \\
    +\infty & \text{, } x \in \mathbb{R} \setminus \mathbb{Z}.
    \end{cases}$$ We also set $\hat{g}: \mathbb{R}^n \to [0,+\infty], \ \hat{g}(x_1,\ldots,x_n)=\sum\limits_{i=1}^n \hat{g_i}(x_i)$ for every $x \in \mathbb{R}^n$. Notice now that $$\tilde{g}=(\hat{g})^{\ast \ast}=\sum\limits_{i=1}^{n} (\hat{g_i})^{\ast \ast}=\sum\limits_{i=1}^{n}\tilde{g_i}.$$  This concludes the proof. 
\end{proof}

Now we can give the definition of discrete log-concavity on $\mathbb{Z}^n$.

\begin{definition}[Discrete log-concave random vectors]\label{def:discrete-log-vectors} An integer valued random vector $X$ with p.m.f $p: \mathbb{Z}^n \to [0,1)$ is called  \textit{discrete log-concave}, if ${\rm supp}(p)=\{k \in \mathbb{Z}^n: p(k)>0\} $ is full-dimensional and $g=- \log p$ is convex-extensible. We denote by $f=e^{-\tilde{g}}$ the log-concave extension of $p$ on $\mathbb{R}^n$.

\end{definition}

Note that $f$ is upper-semicontinuous and if $A=\{f>0\}$, then $f$ is continuous on ${\rm int}(A)$ and $$\conv({\rm supp}(p)) \subseteq A \subseteq \overline{\conv}({\rm supp}(p)).$$

\begin{proposition}\label{prop:equivalence} In the case $n=1$, i.e.\ on $\mathbb{Z}$, the two definitions are equivalent.
    
\end{proposition}

\begin{proof}
    Let $X$ be an integer valued random variable with p.m.f $p: \mathbb{Z} \to [0,1).$
    
    If $p$ admits a log-concave extension $f$ on $\mathbb{R}$, then it is clear that its support is a discrete interval and $$p(k)^2=f(k)^2 \geq f(k-1)f(k+1)=p(k-1)p(k+1)$$ for every $k \in \mathbb{Z}$, due to the log-concavity of $f$.
    
    Assume now that $p=e^{-g}$ is a log-concave sequence in the sense of Definition~\ref{def:one-dimension}. We will show that $g$ is convex-extensible and $$h: \mathbb{R} \to [0,+\infty], \quad h(x)\coloneqq g(k)+(x-k)\left(g(k+1)-g(k)\right) , \, x \in [k,k+1]$$ is the convex extension of $g$. Indeed, $h(k)=g(k)$ for every $k \in \mathbb{Z}$, and it is obvious that $h$ is lower-semicontinuous and convex as it is piecewise affine and its slopes $m_k \coloneqq h(k+1)-h(k)=g(k+1)-g(k)$ are non-decreasing, since $p=e^{-g}$ is a log-concave sequence. So, $h \leq \tilde{g}$, which implies that $\tilde{g}(k)=g(k)$ for every $k \in \mathbb{Z}$. Thus, $g$ is convex-extensible. Finally, let $a(x)=bx+d$ be an affine function with $a(k) \leq g(k)$ for every $k \in \mathbb{Z}$. If $x \in [k,k+1]$ for some $k \in \mathbb{Z}$, then 
    \begin{align*}   
    a(x)&=a(\lambda k+(1-\lambda)(k+1))=\lambda a(k)+(1-\lambda)a(k+1) \leq \lambda g(k) +(1-\lambda)g(k+1)\\
    &= \lambda h(k) +(1-\lambda)h(k+1)=h(\lambda k +(1-\lambda)(k+1))=h(x),
    \end{align*}
    for some $\lambda \in [0,1]$, since $h$ is affine on $[k,k+1]$. Taking supremum over $a$ yields the result.
\end{proof}

Next we discuss the integrability of the log-concave extension $f$. We begin with some almost trivial cases. A proof of the next Proposition can also be found in \cite[Proposition 5.1]{BMaMe}.

\begin{proposition}\label{prop:one-dim-integr}
    Let $X$ be a discrete log-concave random variable in $\mathbb{Z}$ with p.m.f $p : \mathbb{Z} \to [0,1)$ and $f$ be its log-concave extension. Then $f$ is integrable.
\end{proposition}

\begin{proof}
We know that $f$ is unimodal, as a log-concave function. Since $p$ is a p.m.f we can conclude that there exists $x_0 \in \mathbb{R}$ such that $f$ is non-decreasing on $(-\infty,x_0]$ and non-increasing on $[x_0,+\infty)$. By the formula of $f$ given in the proof of Proposition~\ref{prop:equivalence}, we can assume that there exists $k_0 \in \mathbb{Z}$ such that $\max f=f(k_0)$. Then, $$\int_{\mathbb{R}} f(x) \, dx= \sum\limits_{k=-\infty}^{+\infty} \int_{k}^{k+1} f(x) \, dx \leq \sum\limits_{k=-\infty}^{k_0-1} p(k+1) + \sum\limits_{k=k_0}^{+\infty}p(k) \leq 2 <+\infty.$$ 
\end{proof}

\begin{proposition}
    Let $X$ be a random vector in $\mathbb{Z}^n$ with independent and discrete log-concave coordinates. Then, $X$ is a discrete log-concave random vector in $\mathbb{Z}^n$ and the log-concave extension of its p.m.f is integrable.
\end{proposition}

\begin{proof}
    Let $X=(X_1,\ldots,X_n)$ with p.m.f $p :\mathbb{Z}^n \to [0,1)$. Let also $p_i : \mathbb{Z} \to [0,1)$ be the p.m.f of the $i-$coordinate, and $f_i$ be its log-concave extension on $\mathbb{R}$ for every $1 \leq i \leq n$. Proposition~\ref{prop:one-dim-integr} shows that $f_i$ is integrable for every $1 \leq i \leq n$. By Lemma~\ref{lem:separability} we easily deduce that $-\log p=-\sum\limits_{i=1}^n \log p_i$ is convex-extensible and if $f$ is the log-concave extension of $p$, then $$f(x_1,\ldots,x_n)= \prod_{i=1}^{n} f_i(x_i)$$ for every $x \in \mathbb{R}^n$. The integrability of $f$ follows.
\end{proof}

\begin{proposition}\label{prop:restiction-log-concave}
    Let $F: \mathbb{R}^n \to [0,+\infty)$ be log-concave, upper-semicontinuous and integrable, with ${\rm supp}(F) \cap \mathbb{Z}^n$ full-dimensional. Define $p: \mathbb{Z}^n \to [0,1), \, p(k)=\frac{F(k)}{\sum_{\ell \in \mathbb{Z}^n}F(\ell)}$ for every $k \in \mathbb{Z}^n$. Then $p$ is a discrete log-concave p.m.f and its log-concave extension $f$ is integrable. 
\end{proposition}

\begin{proof} Since $F$ is integrable and log-concave, there exist $A,B>0$ such that $F(x) \leq Ae^{-B\|x\|_2}$. Therefore, $\sum_{\ell \in \mathbb{Z}^n}F(\ell)<+\infty$ and $p$ is indeed a p.m.f.

Now we can write $p=e^{-g}$ and  $\frac{F}{\sum_{\ell \in \mathbb{Z}^n}F(\ell)}=e^{-G}$ where $G:\mathbb{R}^n \to \mathbb{R} \cup \{+\infty\}$ is a convex, and lower-semicontinuous function. Notice that  $G(k)=g(k)$ for every $k \in \mathbb{Z}^n$. This implies that $g$ is convex-extensible and $G \leq \tilde{g}$. Thus, $$f=e^{-\tilde{g}} \leq e^{-G}=\frac{F}{\sum_{\ell \in \mathbb{Z}^n}F(\ell)}$$ and the integrability of $f$ follows from the integrability of $F$.
\end{proof}

\begin{remark}
    If $K$ is a convex body in $\mathbb{R}^n$ and $\mathbb{Z}^n \cap K$ is full-dimensional, then by Proposition~\ref{prop:restiction-log-concave} the uniform distribution on $\mathbb{Z}^n \cap K$ is a discrete log-concave distribution in the sense of Definition~\ref{def:discrete-log-vectors}.
\end{remark}

We proceed now to more general cases. We first need some lemmas. The following lemma is a direct consequence of \cite[Theorem VII]{Cassels} :

\begin{lemma}[infinite approximation]\label{lem:dioph}
    For every $d_1,\ldots,d_n \in \mathbb{R}$ and every $\varepsilon>0$ there exist infinitely many $q,p_1,\ldots,p_n \in \mathbb{Z}$ such that $|qd_i-p_i| < \varepsilon \ $ for every $i=1,\ldots,n$.
\end{lemma}

\begin{lemma}\label{lem:infinite-lattice}
    Let $K \subseteq \mathbb{R}^n$ be closed, convex, unbounded with non-empty interior. If ${\rm int}(K) \cap \mathbb{Z}^n \neq \emptyset$, then it is an infinite set.
\end{lemma}

\begin{proof}
    Assume without loss of generality that $0 \in {\rm int}(K)$. We define the recession cone of $K$: $${\rm rec}(K)=\{ d \in \mathbb{R}^n: x+td \in K \ \text{ for every  } x \in K \text{ and } t \geq 0 \}. $$
    Since $K$ is non-empty, convex and unbounded, by \cite[Theorem 8.4]{Rock} , we have that there exists some $d=(d_1,\ldots,d_n) \in {\rm rec}(K) \setminus \{0\}$. Therefore, $td \in K$ for every $t \geq 0$.  Now, since $0 \in {\rm int}(K)$ , there exists $r \in (0,|d|/2)$ such that $2rB_2^n \subseteq K$. By convexity of $K$ we get that $B(td,r) \subseteq K$ for every $t \geq 0$. So, by Lemma~\ref{lem:dioph} there exist infinitely many $q,p_1,\ldots,p_n \in \mathbb{Z}$ such that
    $$ |qd_i-p_i| < \frac{r}{\sqrt{n}}, $$ for every $i=1,\ldots,n$. This implies that there are infinitely many $q \in \mathbb{N}$ for which $B(qd,r) \cap \mathbb{Z}^n \neq \emptyset$. The choice of $r>0$ makes those balls disjoint, which concludes the lemma.
\end{proof}

\begin{proposition}\label{prop:bounded}
Let $X$ be a discrete log-concave random vector in $\mathbb{Z}^n$ with p.m.f $p:\mathbb{Z}^n \to [0,1)$ and let $f$ be its log-concave extension to $\mathbb{R}^n$. If  ${\rm supp}(p)$ is bounded, then $f$ is integrable.
\end{proposition}

\begin{proof}
Since $f$ is log-concave, there exist $c>0$ and $v \in \mathbb{R}^n$, such that $f(x)\leq c e^{\langle x,v \rangle}$ for every $x \in \mathbb{R}^n$. Moreover,  ${\rm supp}(f)=\overline{\{f>0\}}$ is compact, as ${\rm supp}(p)$ is bounded. So, it is not hard to see that $$\int\limits_{\mathbb{R}^n} f(x) \, dx \leq c \int_{\mathrm{supp}(f)} e^{\langle v,x\rangle}\,dx < +\infty.$$
    
\end{proof}

\begin{proposition}\label{prop:unbounded}
    Let $X$ be a discrete log-concave random vector in $\mathbb{Z}^n$ with p.m.f $p:\mathbb{Z}^n \to [0,1) $ and let $f$ be its log-concave extension. If ${\rm supp}(p)$ is unbounded and ${\rm int}(\conv({\rm supp}(p))) \cap \mathbb{Z}^n \neq \emptyset$, then $f$ is integrable.
\end{proposition}

\begin{proof}
    We may assume without loss of generality that $0 \in {\rm int}(\conv({\rm supp}(p)))$ and therefore $0 \in {\rm int}(\{f>0\})$. We now set $t=f(0)/2 \in (0,1)$ and $C=\{x:f(x)\geq t\}$. By continuity of $f$ at $0$, there exists some $r>0$ such that $rB_2^n \subseteq C$. We also set $K=\{x:f(x) \geq t/e \}$ which is closed and convex by upper-semicontinuity and log-concavity of $f$. We will show that $K$ is bounded. Assuming otherwise, since $0 \in {\rm int}(K)$, by Lemma~\ref{lem:infinite-lattice},  $\tilde{K}={\rm int}(K) \cap \mathbb{Z}^n$ is infinite. Then for 
every $m\in \tilde{K}$ we have $p(m)=f(m)\ge t/e>0$. Hence
\[
\sum_{m\in\mathbb{Z}^n}p(m)\ \ge\ \sum_{m\in \tilde{K}}\frac{t}{e}\ =\ +\infty,
\]
a contradiction to $p$ being a p.m.f .
Therefore, $K \subseteq \frac{R}{2}B_2^n$ for some $R>2r>0$.

     So, if $x \in \mathbb{R}^n$ with $\|x\|_2>R$ we can write $$\frac{Rx}{\|x\|_2}=\frac{\|x\|_2-R}{\|x\|_2-r}\frac{rx}{\|x\|_2}+\frac{R-r}{\|x\|_2-r}x,$$ as a convex combination, with $f(Rx/\|x\|_2) < t/e$ and $f(rx/\|x\|_2) \geq t$. By log-concavity of $f$ we get $$f(x) \leq t e^{- \frac{\|x\|_2-R}{R-r}} \leq C e^{-\|x\|_2 / R},$$ for every $\|x\|_2 > R$.

    On the other hand, if $\|x\|_2\leq R$ then for every $y \in \frac{x}{2}+\frac{r}{2}B_2^n$ we have $$f(y) \geq \sqrt{f(x)f(2y-x) } \geq \sqrt{tf(x)}.$$ Integrating now over $y$, and taking into account that $\frac{x}{2}+\frac{r}{2}B_2^n \subseteq RB_2^n$, we see that $$\int\limits_{RB_2^n} f(y) \, dy \geq \sqrt{tf(x)} \  \vol_n\left(\frac{r}{2}B_2^n\right),$$ which yields that $f(x) \leq M$ for some $M>0$, as $f$ is log-concave and $RB_2^n$ is compact.

    Combining the above, we see that there are some $A,B>0$ such that $f(x) \leq Ae^{-B\|x\|_2}$ for every $x \in \mathbb{R}^n$, which directly concludes the proof.
\end{proof}

However, in $\mathbb{R}^2$ we can ensure that the log-concave extension of every  p.m.f of a discrete log-concave random vector is integrable. To this end, we need the following lemmas:

\begin{lemma}[Kronecker's theorem]\label{lem:kronecker} 
If $a \in \mathbb{R} \setminus \mathbb{Q}$, then $(na-[na])_{n \in \mathbb{N}}$ is dense in $[0,1)$.
\end{lemma}

\begin{lemma}\label{lem:R^2}
    Let $K \subseteq \mathbb{R}^2$ be convex, closed, unbounded, with non-empty interior. If $K \cap \mathbb{Z}^2 \neq \emptyset$, then $K \cap \mathbb{Z}^2$ is infinite.
    
\end{lemma}

\begin{proof}
    Without loss of generality we may assume that $0 \in K$. If $0 \in {\rm int}(K)$, then  Lemma~\ref{lem:infinite-lattice} gives the result. So, we can suppose that $0 \in {\rm bd}(K)$. Since $K$ is unbounded, we can find some $d \in {\rm rec}(K) \setminus \{0\}$. So, we have $$B=\{td: t\geq 0 \} \subseteq K.$$ If ${\rm span} \{d\} \cap\mathbb{Q}^2 \neq \{0\}$, then it is clear that $B \cap \mathbb{Z}^2$ is infinite. Assume now otherwise and let $v \in {\rm int}(K) \setminus {\rm span}\{d\}$. So, we also have $$C=\{v+td: t \geq 0\} \subseteq K.$$ By convexity of $K$ it is enough to show that $\conv\{B,C\} \cap \mathbb{Z}^2$ is infinite. Now, one can see that it suffices to show that for every $a \in \mathbb{R} \setminus \mathbb{Q}$ and $0<b<1$ there are infinitely many solutions $(x,y) \in \mathbb{N}^2$ of $$0 \leq y-ax <b.$$  Indeed, by Lemma~\ref{lem:kronecker} we can find a strictly increasing sequence $(n_k)_{k \in \mathbb{N}}$ of natural numbers such that $n_k a-[n_ka] \in (1-b,1)$ for every $k \in \mathbb{N}$. Thus,  $(n_k,[n_ka]+1) \in \mathbb{N}^2$ satisfies the above inequality for every $k \in \mathbb{N}$.
\end{proof}

\begin{proposition}\label{prop:Z^2}
    Let X be a discrete log-concave random vector in $\mathbb{Z}^2$ with p.m.f $p: \mathbb{Z}^2 \to [0,1)$ and let $f$ be its log-concave extension. Then $f$ is integrable.
\end{proposition}

\begin{proof}
    By Proposition~\ref{prop:bounded}, we may assume that ${\rm supp}(p)$ is unbounded. Furthermore, without loss of generality we can suppose that $f(0)>0$. Since ${\rm supp}(p)$ is also full-dimensional, there exists some $x_0 \in \operatorname{int}(\{f>0\})$. Select $m>0$ such that $f(x_0)>f(0)/m$ and set $t=f(0)/m$, $C=\{x:f(x) \geq t\}$ and $K=\{x:f(x) \geq t/e\}$. By continuity of $f$ at $x_0$ there exists $r>0$ such that $x_0+rB_2^n \subseteq C$. Also, $K$ is convex, closed with non-empty interior, and $0 \in K$. So, by Lemma~\ref{lem:R^2}, $K$ cannot be unbounded, as $p$ is a p.m.f. It follows that $K \subseteq x_0+\frac{R}{2}B_2^n$ for some $R>2r>0$. Then, translating by $-x_0$ and working with $f_1(\cdot)=f(\cdot +x_0)$ as in the proof of Proposition~\ref{prop:unbounded}, we get the result.
\end{proof}

\begin{remark}\label{koloun}
    If Lemma~\ref{lem:R^2} was valid for every $n \geq 2$, then in the same way, we could show the integrability of the log-concave extension in full generality. However, for $n\geq 3$ the lemma fails. Indeed, if we consider $$K=\left\{(x,y,z): \left(x-\sqrt{2}z\right)^2+\left(y-\frac{1}{3}\right)^2 \leq \frac{1}{9}\right\} \subseteq \mathbb{R}^3,$$ then $K$ is convex, closed, unbounded, with non-empty interior and $K \cap \mathbb{Z}^3=\{0\}$. Moreover, in $\mathbb{R}^4$ the sets $K_m=K \times [1,m], m \in \mathbb{N}$ have the same properties and can only contain finitely many integer points, which makes the proof of Proposition~\ref{prop:Z^2} not work.
\end{remark}

Now we show that the integrability of the log-concave extension implies the finiteness of $\Lambda_{X}^\ast$ moments.

\begin{proof}[Proof of Theorem~\ref{thm:lambda-star-moments}]
Let $f=e^{-\tilde{g}}$ , where $\tilde{g}:\mathbb{R}^n\to[0,+\infty]$ is a convex function. Due to integrability of $f$, there exist $A,B>0$, such that $f(x) \leq Ae^{-B\|x\|_2}$ for every $x \in \mathbb{R}^n$.

 Now, if $k \in \mathbb{Z}^n$ and $\xi \in \mathbb{R}^n$, we have $$e^{\Lambda_{X}(\xi)}= \sum\limits_{m \in \mathbb{Z}^n} p(m)e^{\langle m,\xi \rangle} \geq p(k)e^{\langle k,\xi \rangle}. $$
Therefore, $$\Lambda_{X}^{\ast}(k)=\sup\limits_{\xi \in \mathbb{R}^n} \{ \langle k, \xi \rangle - \Lambda_{X}(\xi) \} \leq \log \left( \frac{1}{p(k)} \right).$$
So, if $q \geq 1$ $$\mathbb{E}[(\Lambda_{X}^{\ast})^q]= \sum\limits_{k \in \mathbb{Z}^n} (\Lambda_X^{\ast}(k))^q p(k) \leq \sum\limits_{k \in \mathbb{Z}^n} p(k) \log^q\left( \frac{1}{p(k)} \right)= \sum\limits_{k \in \mathbb{Z}^n} \tilde{g}^q(k)e^{-\tilde{g}(k)}$$
Since $\lim\limits_{x\to +\infty} x^qe^{-x/2}=0$ and $\lim\limits_{\|x\|_2\to +\infty} \tilde{g}(x)=+\infty$,  there exists $R>0$ such that $\tilde{g}^q(x)\leq e^{\tilde{g}(x)/2}$ for every $\|x\|_2>R$. Thus,  $$\mathbb{E}[(\Lambda_{X}^{\ast})^q] \leq N +\sum\limits_{\substack{k \in \mathbb{Z}^n \\ \|k\|_2>R}}e^{-\tilde{g}(k)/2}=N+ \sqrt{A}\sum\limits_{\substack{k \in \mathbb{Z}^n \\ \|k\|_2>R}}e^{-B\|k\|_2/2} < +\infty.$$
\end{proof}

\section{Weak threshold for product measures}

Let $\mu$ be an even  Borel  probability measure on $\mathbb{R}$ and let $$x^\ast=x^\ast(\mu)=\sup \{ x \in \mathbb{R}: \mu([x,+\infty))>0\} \in (0,+\infty]$$ be its right endpoint. For every $n \in \mathbb{N}$ we set $\mu_n:=\mu\otimes\cdots\otimes\mu$ ($n$-times). We will show that for certain classes of measures $\mu$, the sequence $(\mu_n)_{n \in \mathbb{N}}$ exhibits a weak threshold. The method we employ has been studied extensively in \cite{GaGia} and \cite{Paf}. We therefore only sketch the argument, emphasizing the main ideas and the aspects that are specific to the cases under consideration. A key ingredient is the following lemma.

\begin{lemma}\label{lem:DFM}
Let $\mu$ be a Borel probability measure on $\mathbb{R}$ and $K_N$ be the convex hull of $N$ random vectors independently chosen according to $\mu_n$.\\
{\rm (a)} For every Borel subset $A$ of $\mathbb{R}^n$,
\begin{equation}\label{eq:DFM-lemma-upper}
\E_{\mu_n^N}[\mu_n(K_N)]\ls \mu_n(A)+N\cdot \sup_{x\in A^c}q_{\mu_n}(x).
\end{equation}
{\rm (b)} In addition, if $p_\mu=\max \{\mu(\{x\}): x \in \mathbb{R}\}<1$, then
\begin{equation}\label{eq:DFM-lemma-lower}
\E_{\mu_n^N}[\mu_n(K_N)]\gr \mu_n(A)\left(1-\binom{N}{n}p_\mu^{N-n}-2\binom{N}{n}\left(1-\inf_{x\in A}q_{\mu_n}(x)\right)^{N-n}\right).
\end{equation}
\end{lemma}

\subsection{Threshold for product measure with mass at the endpoint}

In this subsection we identify and provide a missing justification that is necessary for the thresholds in \cite{DFM} and \cite{GaGia}. As an application, we show a weak threshold for even compactly supported product measures with mass at the endpoint, which is the case that was not covered in \cite{Paf}.

Hence we assume that $x^\ast <+\infty$ and we set $p^\ast=\mu(\{x^\ast\})$. To begin with, note that the m.g.f $\phi_{\mu}(\xi)$ of $\mu$ converges for every $\xi \in \mathbb{R}$, since $\mu$ is compactly supported. Moreover, if $\Lambda_\mu=\log \phi_\mu$ we have the following lemma (for a proof see \cite[Lemma 2.3]{GaGia}).

\begin{lemma}\label{lem:Lambda-property} $\Lambda_\mu': \mathbb{R} \to (-x^\ast,x^\ast) $ is strictly increasing and surjective. In particular, $$\lim_{\xi\to \pm \infty}\Lambda_{\mu}'(\xi)=\pm x^{\ast} .$$
\end{lemma}

Now we can define $h:(-x^\ast,x^\ast) \to \mathbb{R}$ by $h=(\Lambda_{\mu}')^{-1}$, which is a strictly increasing and $C^{\infty}$ function. Consider also the Cram\'{e}r transform $\Lambda_{\mu}^{\ast}$. In the next lemma we collect the basic properties of $\Lambda_{\mu}^{\ast}$ (for a proof see \cite[Proposition 2.12]{GaGia2006} and \cite[Lemma 2.8]{GaGia}).

\begin{lemma}\label{lem:star-prop} The following hold true:
\begin{enumerate}
\item[{\rm (i)}]  $\Lambda_{\mu}^{\ast }\gr 0$, $\Lambda_{\mu}^{\ast }(0)=0$ and
$\Lambda_{\mu}^{\ast }(x)=+\infty$ for $x\in\mathbb{R}\setminus [-x^{\ast} ,x^{\ast}
]$.
\item[{\rm (ii)}] For every $x\in (-x^\ast,x^{\ast})$ we have
$\Lambda_{\mu}^{\ast }(x)=\xi x-\Lambda_{\mu}(\xi)$ if and only if $\Lambda_{\mu}'(\xi)=x;$ hence
$$\Lambda_{\mu}^{\ast }(x)=x h(x)-\Lambda_{\mu}(h(x)) \qquad\text{for}\ x\in(-x^\ast,x^{\ast}) .$$
\item[{\rm (iii)}] $\Lambda_{\mu}^{\ast }$ is a strictly convex
$C^\infty$ function on $(-x^\ast,x^{\ast}),$ and $$(\Lambda_{\mu}^{\ast })'(x)=h(x).$$
\item[{\rm (iv)}] $\Lambda_{\mu}^{\ast }$ attains its unique minimum on
$(-x^\ast,x^\ast)$ at $x=0$.
\item[{\rm (v)}] $\Lambda_{\mu }^{\ast }$ admits a continuous extension to $\pm x^\ast$ by setting $\Lambda_{\mu}^\ast(\pm x^\ast)=-\log p^\ast$.
\end{enumerate}
\end{lemma}

\begin{remark}
    We note that every such measure $\mu$ satisfies the $\Lambda^\ast$-condition. In other words, we have $$\lim\limits_{x\uparrow x^{\ast}}\dfrac{-\log \mu([x,+\infty))}{\Lambda_{\mu}^{\ast }(x)}=1.$$
\end{remark}

We will now estimate $\varrho_{i}(\mu_n,\delta), \ i=1,2$. Notice that our definition for $\varrho_{i}(\mu_n,\delta)$ gives the ones from \cite{Paf} multiplied by $n$.  We start with the upper threshold. The proof of the next theorem is based on Lemma~\ref{lem:DFM}\,(a). It is identical  to the one in \cite[Theorem 3.1]{Paf} and holds true for every even Borel probability measure on $\mathbb{R}$.

\begin{theorem}\label{th:upper-mun}
For every $\delta\in \left(0,\tfrac{1}{2}\right)$ there exist $c(\mu,\delta)>0$ and $n_0(\mu,\delta )\in {\mathbb N}$ such that
$$\varrho_1(\mu_n ,\delta )\gr \left(1-\frac{c(\mu,\delta)}{\sqrt{n}}\right)\mathbb{E}_{\mu_n}[\Lambda_{\mu_n}^{*}].$$
\end{theorem}

Before we proceed to the lower threshold, we need some definitions and lemmas.

\begin{definition}[Normal and tangent cones] Let $K \subseteq \mathbb{R}^n$ be  convex and closed. For every $x_0 \in {\rm bd}(K)$ we define the normal cone of $K$ at $x_0$ as $$N_K(x_0)= \{v \in \mathbb{R}^n: \langle v,x-x_0 \rangle \leq 0 \text{ for every } x \in K\} $$ and the tangent cone of $K$ at $x_0$ as $$T_K(x_0)=\overline{\{ y \in \mathbb{R}^n: \exists t>0 \text{ such that } x_0+ty\in K\}} .$$
\end{definition}

These two cones are connected in the following way:

\begin{lemma}\label{lem:cones}
    Let $K \subseteq \mathbb{R}^n$ be convex and closed. For every $x_0 \in {\rm bd}(K)$ we have that $N_K(x_0)=\{v \in \mathbb{R}^n : \langle v,z \rangle \leq 0 \text{ for every } z \in T_K(x_0) \}.$
\end{lemma}

\begin{proof}
    Let $v \in N_K(x_0)$ and $z \in T_K(x_0)$. We can find a sequence $(t_k)_{k \in \mathbb{N}}$ of positive real numbers and a sequence $(y_k)_{k \in \mathbb{N}} $ in $\mathbb{R}^n$ such that $y_k \to z$, as $k \to \infty$ and $x_0+t_ky_k \in K$ for every $k \in \mathbb{N}$. Since $v \in N_K(x_0)$ we get $\langle v,x_0+t_ky_k-x_0 \rangle \leq 0$ for every $k \in \mathbb{N}$, which gives $\langle v,y_k \rangle \leq 0 $ for every $k \in\mathbb{N}$. Taking $k \to \infty$, we deduce that $\langle v,z \rangle \leq 0$.

     For the other inclusion, let $v \in \mathbb{R}^n$ such that $\langle v, z \rangle  \leq 0$ for every $z \in T_K(x_0)$. Let $x \in K$ and $t \in (0,1)$. We can write $$x_0+t(x-x_0)=tx+(1-t)x_0 \in K,$$ which means that $x-x_0 \in T_K(x_0)$. We conclude that $\langle v, x-x_0 \rangle \leq 0$, and consequently $v \in N_K(x_0)$.
\end{proof}

In what follows, for any $r>0$ we set $$B_r(\mu_n)= \{x \in \mathbb{R}^n: \Lambda_{\mu_n}^\ast(x) \leq r \} \subseteq [-x^\ast,x^\ast]^n,$$ which is a compact set, by the continuity of $\Lambda_{\mu_n}^\ast$ on $[-x^\ast,x^\ast]^n$. We will use Lemma~\ref{lem:cones} in order to extract a formula for all the supporting hyperplanes of $B_r(\mu_n)$ at points lying on the faces of the cube $[-x^\ast,x^\ast]^n$.

\begin{lemma}\label{lem:hyper-bd} Let $x=(x_1,\ldots,x_n) \in [0,x^{\ast}]^n$ with $0 \leq x_1 \leq\cdots \leq x_{k-1}<x_k=\cdots=x_n=x^{\ast}$ for some $1 \leq k \leq n$. If $\Lambda_{\mu_n}^\ast(x)=r$, then $$N_{B_r(\mu_n)}(x)= \{v \in \mathbb{R}^n :v_1=\cdots=v_{k-1}=0, v_{k},\cdots,v_n \geq 0 \}.$$

\end{lemma}

\begin{proof} We will show that $$T_{B_r(\mu_n)}(x)= \{y \in \mathbb{R}^n : y_{k},\ldots,y_n \leq 0 \}$$
and thus by Lemma~\ref{lem:cones} and a simple calculation, the proof is complete.

 To begin with, let $y \in \mathbb{R}^{n}$ with $y_i>0$ for some $k \leq i \leq n$. Then for every $t>0$ we have that $(x+ty)_i>x^{\ast}$, which means that $\Lambda_{\mu}^{\ast}((x+ty)_{i})=+\infty$ and therefore $y \notin \{ y \in \mathbb{R}^n: \exists t>0 \text{ such that } x+ty\in B_r(\mu_n) \}$. So, we have $$\{ y \in \mathbb{R}^n: \exists t>0 \text{ such that } x+ty\in B_r(\mu_n) \} \subseteq \{y \in \mathbb{R}^n : y_{k},\ldots,y_n \leq 0 \}$$
and since the right-hand set is closed we get $$T_{B_r(\mu_n)}(x) \subseteq\{y \in \mathbb{R}^n : y_{k},\ldots,y_n \leq 0 \}.$$
Let now $y \in \mathbb{R}^n$ with $y_{k},\ldots,y_n<0$. We want to find $t>0$ such that $$\Lambda_{\mu_{n}}^{\ast}(x+ty) \leq \Lambda_{\mu_{n}}^{\ast}(x).$$
Equivalently, we need to find some $t>0$ such that:
$$\sum\limits_{i=1}^{k-1} \frac{\Lambda_{\mu}^{\ast}(x_i+ty_i)-\Lambda_{\mu}^{\ast}(x_i)}{t} \leq \sum\limits_{i=k}^{n} \frac{\Lambda_{\mu}^{\ast}(x^{\ast})-\Lambda_{\mu}^{\ast}(x^{\ast}+ty_i)}{t}.$$
But taking limits $t \to 0^{+}$, the left-hand side remains finite and the  right-hand side becomes $+\infty$, by Lemma~\ref{lem:Lambda-property}, the definition of $h$ and Lemma~\ref{lem:star-prop} ${\rm (iii)}$. So, we can find the desired $t>0$. Therefore, we have shown that $$\{y \in \mathbb{R}^n: y_k,\ldots,y_n <0\} \subseteq \{ y \in \mathbb{R}^n: \exists t>0 \text{ such that } x+ty\in B_r(\mu_n) \}.$$
By taking closures we get 
\[
\{y\in\mathbb{R}^n:\ y_k,\ldots,y_n\le 0\}\subseteq T_{B_r(\mu_n)}(x),
\]
which concludes the proof.
\end{proof}

Now, for the lower threshold, the main technical tool is the following theorem, whose proof contains an additional detail, compared to the original argument.

\begin{theorem}
    For every $\zeta >0,$ there exists $n_0(\mu,\zeta)\in\mathbb{N},$ depending only on $\zeta$
and $\mu$, such that for all $r>0$ and all $n\gr n_0(\mu,\zeta)$ we have 
$$\inf_{x\in B_r(\mu_n)}q_{\mu_n }(x)\gr\exp (-(1+\zeta)r-2\zeta n).$$
\end{theorem}

\begin{proof}
Let $x\in B_r(\mu_n)$ and $H_1^+$ be a closed half-space with $x\in H_1$. There exists
$u\in {S^{n-1}}$ such that $H_1^+=\{y\in {\mathbb R}^n:\langle y-x,u\rangle\gr 0\}$. Consider
the function $s:B_r(\mu_n)\to {\mathbb R}$, $s(w)=\langle w, u\rangle$. Since $s$ is continuous and
$B_r(\mu_n)$ is compact, $s$ attains its maximum at some point $z\in B_r(\mu_n)$. Define
$H^+=\{y\in {\mathbb R}^n:\langle y-z,u\rangle\gr 0\}$. Then, $z\in H$ and
for every $w\in B_r(\mu_n)$ we have $\langle w,u\rangle\ls \langle z,u\rangle $, which shows that
$H$ supports $B_r(\mu_n)$ at $z$. Moreover, $H^+\subseteq H_1^+$ and hence
$\mu_n(H^+)\ls \mu_n (H_1^+)$. This shows that $\inf\{q_{\mu_n}(x):x\in B_r(\mu_n)\}$ is attained
for some closed half-space $H^+$ whose bounding hyperplane supports $B_r(\mu_n )$.
Therefore, for the proof of the theorem it suffices to show that given $\zeta >0$ we may find
$n_0(\mu,\zeta )$ so that if $n\gr n_0(\mu,\zeta)$ then
\begin{equation}\label{eq:main}\mu_n(H^+)\gr \exp (-(1+\zeta)r-2\zeta n)\end{equation}
for any closed half-space $H^+$ whose bounding hyperplane supports
$B_r(\mu_n )$.

Let $H^+$ be such a half-space that supports $B_r(\mu_n)$ at some $x=(x_1,\ldots,x_n) \in [-x^{\ast},x^{\ast}]^n$. By symmetry and independence, we can also assume that $x
\in [0,x^{\ast}]^n$. We distinguish two cases:

\noindent \textit{Case 1}: If $x \in [0,x^{\ast})^n$, then $\Lambda_{\mu_n}^{\ast}(x)=r $ and $$H^+=\{ y \in \mathbb{R}^n: \langle \nabla \Lambda_{\mu_n}^\ast(x), y-x \rangle \geq 0 \},$$ since $\Lambda_{\mu_n}^\ast$ is $C^{\infty}$ at $x$ and therefore we have that $N_{B_r(\mu_n )}(x)=\{ \lambda \nabla \Lambda_{\mu_n}^\ast(x): \lambda \geq 0 \}$. The proof in this case goes exactly as in \cite[Theorem 3.3]{Paf}.

\noindent \textit{Case 2}: Assume now, without loss of generality, that there exists $1 \leq k \leq n $ such that $0 \leq x_1 \leq \cdots \leq x_{k-1}<x_k=\cdots=x_n=x^\ast$. Let  $\Lambda_{\mu_n}^\ast(x)=r' \leq r$. Then, since $B_{r'}(\mu_n) \subseteq B_r(\mu_n)$, $H$ also supports $B_{r'}(\mu_n)$ at $x$. So, by Lemma~\ref{lem:hyper-bd} there exist $v_k,\ldots,v_n \geq 0$ such that $$H^+=\{y \in \mathbb{R}^n: \sum\limits_{i=k}^{n}v_i(y_i-x^\ast) \geq 0\}.$$ Recall that $p^\ast=\mu(\{x^\ast\})$. We write 
\begin{align*} \mu_n(H^+) &= \mu_n(\{y \in [-x^\ast,x^\ast]^n: \sum\limits_{i=k}^{n}v_i(y_i-x^\ast) \geq 0\}) 
    \\
    &\geq \mu_n(\{y \in [-x^\ast,x^\ast]^n: y_k=\cdots=y_n=x^\ast \})=(p^\ast)^{n-k+1}
    \\
    &=e^{-(n-k+1)\Lambda_{\mu}^\ast(x^\ast)}= \exp \left(- \sum\limits_{i=k}^{n} \Lambda_{\mu}^\ast(x_i) \right) \geq \exp \left(-(1+\zeta) \sum\limits_{i=1}^{n} \Lambda_{\mu}^\ast(x_i) -2 \zeta n \right)
    \\
    &\geq\exp(-(1+\zeta)r-2\zeta n),
\end{align*} since $\mu_n$ is a product measure and $\Lambda_{\mu}^\ast \geq 0$ .
\end{proof}

We are now able to provide an upper bound for $\varrho_2(\mu_n,\delta)$ using Lemma~\ref{lem:DFM}\,(b) (see also \cite[Theorem 3.9]{Paf}).

\begin{theorem}\label{th:lower-mun}
For any $\delta\in \left(0,\tfrac{1}{2}\right)$ and $\varepsilon\in (0,1)$ we can find $n_0(\mu,\delta,\epsilon )$ such that
$$\varrho_2(\mu_n ,\delta )\ls \left(1+\varepsilon \right)\mathbb{E}_{\mu_n}[\Lambda_{\mu_n}^{*}]$$
for all $n\gr n_0(\mu,\delta,\varepsilon )$.
\end{theorem}

Combining Theorem~\ref{th:upper-mun} and Theorem~\ref{th:lower-mun} we get the final result:

\begin{theorem}\label{th:final}For any $\delta\in \left(0,\tfrac{1}{2}\right)$
and any $\varepsilon\in (0,1)$ there exists $n_0(\mu,\delta,\varepsilon)$ such that
$$\varrho_1(\mu_n,\delta)\gr (1-\varepsilon){\mathbb E}_{\mu_n}[\Lambda_{\mu_n}^{\ast}]\quad\hbox{and}\quad
\varrho_2(\mu_n,\delta)\ls (1+\varepsilon){\mathbb E}_{\mu_n}[\Lambda_{\mu_n}^{\ast}]$$
for every $n\gr n_0(\mu,\delta,\varepsilon)$. In particular, $(\mu_n)_{n \in \mathbb{N}}$ exhibits a weak threshold, around $({\mathbb E}_{\mu_n}[\Lambda_{\mu_n}^{\ast}])_{n \in \mathbb{N}}$.
\end{theorem}

\subsection{Discrete log-concavity and threshold}

Let now $\mu$  be the law of a symmetric, discrete log-concave random variable $Y$ with p.m.f $(p(k))_{k \in \mathbb{Z}}$. If $Y$ is bounded, Theorem~\ref{th:final} shows that $(\mu_n)_{n \in \mathbb{N}}$ exhibits a weak threshold. Hence we will assume that $Y$ is unbounded, i.e. $x^\ast=+\infty$. In this subsection we provide a characterization of the $\Lambda^\ast$-condition in terms of the p.m.f and, based on this characterization, we derive a threshold for the corresponding discrete log-concave measures.

We set $$S(k)=\mu\left([k,+\infty)\right), \qquad k \in \mathbb{Z}$$ and $$m(x)=-\log \mu\left([x,+\infty)\right), \qquad x \in \mathbb{R}.$$ Notice that $m(k)=-\log S(k)$ for every $k \in \mathbb{Z}$. Moreover, using Markov's inequality we get \begin{equation}\label{eq:m-Lambda}
\Lambda_\mu (\xi) \geq \xi x-m(x),    
\end{equation} for every $x>0$ and every $\xi \geq 0$.
We also set $$r_k= \frac{p(k+1)}{p(k)}, \qquad k \in \mathbb{Z}.$$ By log-concavity we have that $(r_k)_{k \in \mathbb{Z}}$ is a non-increasing sequence. We define $$r^\ast \coloneq\lim_{k \to \infty}r_k=\lim_{k \to \infty} \frac{p(k+1)}{p(k)} \in [0,1).$$ In the next lemma, we relate $r^\ast$ to the convergence interval of the m.g.f $$J_\mu=\{\xi \in \mathbb{R}:\phi_\mu(\xi) <+\infty\}.$$

\begin{lemma}\label{lem:mgf-conv}
Setting $\xi^\ast=-\log r^\ast$, with the convention $\log 0=-\infty$,  we have that $J_\mu=(-\xi^\ast,\xi^\ast)$ and $\lim_{\xi \to \xi^\ast} \Lambda_\mu(\xi)=+\infty$.
\end{lemma}

\begin{proof}
    We will deal only with the right endpoint, as $\mu$ is even. Let $q>r^\ast$. Then, there exists $C>0$ such that $p(k) \leq C q^k$, for every $k \in \mathbb{N}$. It is straightforward to see that $\phi_\mu(\xi) < \infty$ for every $0 \leq \xi <-\log q$. Thus, we can deduce that $\phi_\mu(\xi) < \infty$ for every $0 \leq \xi < -\log r^\ast$.

     If $r^\ast=0$, then the first part of the proof is done. So, suppose that $0<r^\ast<1$. Then, $p(k) \geq p(0)(r^{\ast})^k$ for every $k=0,1,2,\ldots$
     
       As a result, we have 
$$\phi_\mu(-\log r^\ast) \geq \sum\limits_{k=0}^{\infty} p(k) \frac{1}{(r^\ast)^k}=+\infty.$$ Now, if $0 \leq \xi_n \uparrow -\log r^\ast$, then by the monotone convergence theorem we have $$\phi_\mu(\xi_n) \geq \sum\limits_{k=0}^{\infty} p(k) e^{k \xi_n} \Big\uparrow \sum\limits_{k=0}^{\infty} p(k) \frac{1}{(r^\ast)^k}=+\infty,$$ as $n \to \infty$.

         So, it remains to show the second part of the lemma, when $r^\ast=0$. Recall the inequality $$\Lambda_\mu(\xi) \geq \xi x -m(x),$$ for every $x>0$ and $\xi \geq 0$. Letting $\xi \to +\infty$ concludes the proof. 
\end{proof}

The previous lemma implies that Lemma~\ref{lem:Lambda-property} holds for the discrete log-concave measure $\mu$ of this subsection (for a proof see \cite[Lemma 2.3 {\rm (ii)} and {\rm (iii)}]{Paf}). So, we can now define $h:(-x^\ast,x^\ast) \to \mathbb{R}$ by $h=(\Lambda_{\mu}')^{-1}$, which is a strictly increasing and $C^{\infty}$ function. As a result, Lemma~\ref{lem:star-prop} {\rm (i)}-{\rm (iv)} hold true for~$\mu$.

For this method to yield a threshold, the $\Lambda^\ast$-condition is necessary, i.e. $$\lim_{x \to x^\ast } \frac{-\log\mu([x,+\infty))}{\Lambda_{\mu}^\ast (x)}=1.$$ We will now characterize the $\Lambda^\ast$-condition in terms of $(p(k))_{k \in \mathbb{Z}}$. Let $p(k)=e^{-g(k)}$ for every $k \in \mathbb{Z}$.

\begin{theorem}\label{thm:Lambda-star-condition}
    $\mu$ satisfies the $\Lambda^\ast$-condition if and only if $m(k+1)/m(k) \to 1$, as $k \to \infty $, and hence if and only if $$\lim_{k \to \infty}\frac{g(k+1)}{g(k)}=1.$$
\end{theorem}

We first need a lemma that compares the asymptotic behaviour of $g(k)$ and $m(k)$.

\begin{lemma}\label{lem:m-asymptotics}
We have that $S(k) \sim p(k)/(1-r^\ast)$ and consequently $m(k) \sim g(k)$ as $k \to \infty$.    
\end{lemma}

\begin{proof}
If $k \in \mathbb{N}$ we write $$S(k)=\mathbb{P}(Y \geq k)=\sum_{m=k}^{+\infty}p(m)=p(k) \sum_{m=k}^{+\infty} \prod_{\ell=k}^{m-1}r_\ell.$$
Since $r_k \downarrow r^\ast$, as $k \to \infty$, we have $$(r^\ast)^{m-k} \leq \prod_{\ell=k}^{m-1}r_\ell \leq r_k^{m-k},$$ which yields $$\frac{p(k)}{1-r^\ast} \leq S(k) \leq \frac{p(k)}{1-r_k}$$ and thus $$1-\frac{1}{g(k)} \log\left(\frac{1}{1-r_k}\right) \leq \frac{m(k)}{g(k)} \leq 1-\frac{1}{g(k)} \log\left(\frac{1}{1-r^\ast}\right).$$ Letting $k \to \infty$ concludes the proof, since $g(k) \to +\infty$, as $k \to \infty$.
\end{proof}

We will also need a result from \cite{BC}: 
\begin{theorem}\label{thm:L*-cond-char}
Let $\mu$ be a probability measure on $\mathbb{R}$ and let $x^\ast\in(-\infty,+\infty]$ denote the right endpoint of $\mathrm{supp}(\mu)$. Then, 
\[
\lim_{x\to x^\ast} \frac{-\log\mu([x,\infty))}{\Lambda_\mu^\ast(x)}=1
\]
holds if and only if there exists a convex function $V(x)$ such that $\lim_{x\to x^\ast}\frac{-\log\mu([x,\infty))}{V(x)}=1$.
\end{theorem}

\begin{proof}[Proof of Theorem~\ref{thm:Lambda-star-condition}]
We first assume that $g(k+1)/g(k) \to 1$, as $k \to \infty$. By the proof of Proposition~\ref{prop:equivalence}, we can extend $g$ to some $\tilde{g}: \mathbb{R} \to [0,+\infty)$, which is piecewise affine and convex. Notice that since $p$ is non-increasing on $\mathbb{N}$, then $\tilde{g}$ is non-decreasing on $[1,+\infty)$ by its construction. Let now $k \in \mathbb{N}$ and $x \in (k,k+1]$. Observe that $$\frac{m(x)}{\tilde{g}(x)}=\frac{m(k+1)}{\tilde{g}(x)}=\frac{m(k+1)}{g(k+1)}\frac{g(k+1)}{\tilde{g}(x)}.$$ Since $\tilde{g}$ is non-decreasing we get $$\frac{m(k+1)}{g(k+1)} \leq \frac{m(x)}{\tilde{g}(x)}\leq \frac{m(k+1)}{g(k+1)}{\frac{g(k+1)}{g(k)}}.$$ The assumption, Lemma~\ref{lem:m-asymptotics} and Theorem~\ref{thm:L*-cond-char} yield the $\Lambda^\ast$-property, since $\tilde{g}$ is convex.

Suppose now that there exists a convex function $V:\mathbb{R} \to \mathbb{R}$ such that $m(x) \sim V(x)$, as $x \to \infty$. Since $m(k+1)>m(k)$, if $m(k+1)/m(k) \nrightarrow 1$, by passing to a subsequence, we can assume that there exists $L>1$ such that $m(k+1)/m(k) \geq L$ for every $k \in \mathbb{N}$. Set $$\varepsilon =\frac{1}{4}\left(1-\frac{1}{L}\right)>0.$$ Then we can find $x_0>0$ such that $$V(x) \leq \frac{m(x)}{1-\varepsilon},$$ for every $x \geq x_0$. Let now $k \in \mathbb{N}$ with $k \geq x_0$. By the convexity of $V$ we have $$\frac{m\left(k+\frac{1}{2}\right)}{V\left(k+\frac{1}{2}\right)} \geq \frac{2m(k+1)}{V(k)+V(k+1)} \geq \frac{2(1-\varepsilon)}{1+\frac{m(k)}{m(k+1)}} \geq \frac{2(1-\varepsilon)}{1+\frac{1}{L}}=\frac{\frac{3}{2}+\frac{1}{2L}}{1+\frac{1}{L}}>1,$$ which is a contradiction, since $m(x) \sim V(x)$ as $x \to \infty$.
\end{proof}

\begin{remark}
    Theorem~\ref{thm:Lambda-star-condition} stands in contrast to the continuous case, where the $\Lambda^\ast$-condition holds true for every log-concave probability measure on $\mathbb{R}$. The last fact is a direct consequence of Theorem~\ref{thm:L*-cond-char}. We also mention that while an inequality of the form $$\Lambda_{\mu}^\ast(x) \geq (1-\varepsilon) \log\frac{1}{q_{\mu}(x)}+C(\varepsilon)$$  holds true for every $\varepsilon \in (0,1)$, for every discrete log-concave probability measure $\mu$ on $\mathbb{R}$ and for every $x \in {\rm supp}(\mu)$ (which is the correct statement of \cite[Theorem 3.4]{BC}), it cannot be extended to every $x \in \conv({\rm supp}(\mu))$, as it would imply the $\Lambda^\ast$-condition for every discrete log-concave probability measure on $\mathbb{R}$.
\end{remark}

\begin{remark}
    We note that Lemma~\ref{lem:mgf-conv}, Lemma~\ref{lem:m-asymptotics} and Theorem~\ref{thm:Lambda-star-condition} do not essentially use the symmetry of $Y$. Corresponding results can be stated for any discrete log-concave random variable which is not upper/lower bounded.

\end{remark}

So, we can now formulate our theorem regarding the threshold. 

\begin{theorem}
Let $\mu$ be an even, unbounded, discrete log-concave probability measure on $\mathbb{R}^n$ with p.m.f $p(k)=e^{-g(k)}$ for every $k \in \mathbb{Z}$. If $g(k+1)/g(k) \to 1$ as $k \to \infty$, then $(\mu_n)_{n \in \mathbb{N}}$ exhibits a weak threshold around $({\mathbb E}_{\mu_n}[\Lambda_{\mu_n}^{\ast}])_{n \in \mathbb{N}}$.
\end{theorem}

\begin{proof}
    It is a direct consequence of Theorem~\ref{th:upper-mun}, which holds true for any even Borel probability measure $\mu$ on $\mathbb{R}$ and of Theorem~\ref{th:lower-mun}, which holds true for $\mu$ under the $\Lambda^\ast$-condition (for a complete proof see \cite[Theorem 3.9]{Paf}).
\end{proof}

\section{Sharp threshold in lattice settings}
We begin with the prototypical \emph{discrete cube} model, where $\mu_n$ is the uniform probability measure on
$\mathcal V_n=\{0,1\}^n$ and $K_N=\conv\{X_1,\dots,X_N\}$ for i.i.d.\ samples $X_i\sim\mu_n$.
In contrast to approaches that focus only on locating the \emph{phase transition} scale of $N$,
here we compute the entire expectation curve
\[
F_{n,N}(\mu_n)=\E_{\mu_n^N}\,[\mu_n(K_N)]
\]
\emph{exactly}, for every $N$ and $n$.
The key geometric observation is that $K_N$ creates no new cube vertices:
a vertex belongs to the random convex hull if and only if it was sampled, so $F_{n,N}(\mu_n)$ reduces to the coupon--collector
expectation of the number of distinct observed vertices. This yields a closed formula and, as immediate corollaries,
all asymptotic regimes (subcritical, critical, supercritical) in a unified way.

We then turn to the \emph{lattice $p$-ball} $L_{n,r,p}=\Z^n\cap rB_p^n$ equipped with the uniform probability measure $\mu_{n,r,p}$.
This example is qualitatively new compared to the product-type/discrete-cube setting:
even though the samples $X_1,\dots,X_N$ are i.i.d.\ under $\mu_{n,r,p}$, the distribution $\mu_{n,r,p}$ itself
does \emph{not} exhibit coordinate independence (the constraint $\|X\|_p\le r$ couples the coordinates),
so one cannot rely on independence-based arguments that are available for product measures.
Instead, we analyze the normalized lattice-point functional
\[
F_{n,N}(\mu_{n,r,p})=\E_{\mu_{n,r,p}^N}\Big[\frac{|K_N\cap \Z^n|}{|L_{n,r,p}|}\Big]
\]
via a geometric--combinatorial \emph{coupon--collector sandwich}. We embed inside $L_{n,r,p}$ a large structured layer $A_{n,r,p}$,
whose cardinality 
occupies an asymptotically full
proportion of $L_{n,r,p}$ . The lattice points captured by the
random polytope are then sandwiched between the number of distinct sampled points and the
same quantity augmented by the (rare) lattice points missed inside $A_{n,r,p}$ together with the
negligible complement $L_n\setminus A_n$. Consequently, $F_{n,N}(\mu_{n,r,p})$ exhibits a sharp-threshold behaviour at polynomial scale in $n$.

For both cases we will need the next lemma, which is closely related to the coupon-collector problem.

\begin{lemma}\label{lem:coupon-collector}
    Let $X_1,\ldots,X_N$ be i.i.d.\ random variables uniformly distributed on
$\{v_1,\ldots,v_M\}$. If $D_N= | \{X_1,\ldots,X_N\}|$, then $$\E[D_N]=M(1-(1-1/M)^N).$$
\end{lemma}

\begin{proof}
    For every $1 \leq i \leq M$ we set $Y_i=\mathds{1}_{\{v_i \in \{X_1,\ldots,X_N\} \}}$. Observe that $Y_i \sim \text{Bernoulli}(p)$, where $p=1-(1-1/M)^N$. Notice also that $D_N=\sum_{i=1}^M Y_i$ and we have the proof.
\end{proof}

\subsection{Discrete hypercube $\{0,1\}^n$: complete analysis}

Now let $\mu_n$ be the \emph{uniform } probability on the vertex set
$\mathcal{V}_n:=\{0,1\}^n$; i.e., $\mu_n(\{v\})=2^{-n}$ for all $v\in\mathcal{V}_n$.
Draw $X_1,\dots,X_N$ i.i.d.\ from $\mu_n$  and set $K_N=\conv\{X_1,\dots,X_N\}$. First we will show that the convex hull cannot contain any vertices other than the ones that were sampled.

\begin{lemma}\label{lem:no-new-verts}
Let $S\subseteq\mathcal{V}_n$. Then
\[
\conv(S)\cap \mathcal{V}_n\ =\ S.
\]
Equivalently, a cube vertex $v\in\mathcal{V}_n$ belongs to $\conv(S)$ \emph{if and only if} $v\in S$.
\end{lemma}

\begin{proof}
If $v\in \conv(S) \cap \mathcal{V}_n$ then $v=\sum_{x\in S}\lambda_x x$ with $\lambda_x\ge 0$, $\sum_{x \in S}\lambda_x=1$.
For any coordinate $i$ with $v_i=1$ we have $1=v_i=\sum_{x\in S}\lambda_x x_i$, which forces $x_i=1$ for all $x$ with $\lambda_x>0$.
Similarly, for $v_i=0$ all such $x$ satisfy $x_i=0$. Hence, every $x$ with $\lambda_x>0$ must be equal to $v$ coordinatewise, i.e.\ $x=v$.
Thus $v\in S$. The converse is trivial.
\end{proof}

Let $D_N:=|\{X_1,\dots,X_N\}|$ be the number of \emph{distinct} vertices observed among the $N$ draws.
By Lemma~\ref{lem:no-new-verts},

\begin{equation}\label{eq:cube-measure}
    \mu_n(K_N)\ =\ \frac{D_N}{2^n}.
\end{equation}

Since draws are i.i.d.\ over a finite set of size $M:=2^n$, by Lemma~\ref{lem:coupon-collector} we have
\begin{equation}\label{eq:cube-cc}
\E_{\mu_n^N} [D_N]\ =\ 2^n\left(1-\Big(1-\frac{1}{2^n}\Big)^N\right).
\end{equation}
Combining \eqref{eq:cube-measure} and \eqref{eq:cube-cc} we get:

\begin{theorem}\label{thm:cube-exact}
For every $N,n \in \mathbb{N}$ we have
\begin{equation}\label{eq:cube-exact}
F_{n,N}(\mu_n)\ =\ \E_{\mu_n^N}\,[\mu_n(K_N)]
\ =\ 1-\Big(1-2^{-n}\Big)^{N}.
\end{equation}
\end{theorem}

\begin{corollary}\label{cor:cube-regimes}
\leavevmode
\begin{enumerate}

\item[{\rm (i)}] If $N=o(2^n)$, then $F_{n,N}(\mu_n)\to 0$, as $n \to \infty$.
\item[{\rm (ii)}] If $N\sim \alpha  2^n$, as $n \to \infty$, for some fixed $\alpha>0$, then $F_{n,N}(\mu_n)\to 1-e^{-\alpha}\in(0,1)$, as $n \to \infty$.
\item[{\rm (iii)}] If $N=\omega(2^n)$, then $F_{n,N}(\mu_n)\to 1$, as $n \to \infty$.

\end{enumerate}
In particular, $(\mu_n)_{n \in \mathbb{N}}$ exhibits a sharp threshold around $(T_n)_{n \in \mathbb{N}}=(n\log2)_{n \in \mathbb{N}} $.
\end{corollary}

\begin{proof} Claims {\rm (i)-(iii)} follow from \eqref{eq:cube-exact} and the inequality 
    $$\exp\left(-\frac{Nx}{1-x}\right) \leq (1-x)^N \leq \exp(-Nx)
$$ for every $N \in \mathbb{N}$ and every $x \in (0,1)$, which yield

\begin{equation}\label{eq:cube-threshold}
    1-\exp\left(-\frac{N}{2^n}\right) \leq F_{n,N}(\mu_n) \leq 1-\exp\left(-\frac{N}{2^n-1}\right).
\end{equation}
    Now for the threshold, let $\varepsilon>0$ and $\delta_n=1/n^b$ for every $n \in \mathbb{N}$, for some fixed (but arbitrarily large) $b>0$ . If $N \leq e^{(1-\varepsilon)n\log 2}=2^{(1-\varepsilon)n}$, the right-hand side of \eqref{eq:cube-threshold} gives $$F_{n,N}(\mu_n) \leq 1-\exp\left(-2^{-n \varepsilon/2}\right) \leq 2^{-n\varepsilon/2}< \delta_n ,$$ for every $n \geq n_1(\varepsilon)$, which implies \begin{equation}\label{eq:cube-rho1}
        \varrho_1(\mu_n,\delta_n) \geq (1-\varepsilon)n \log 2.
    \end{equation} On the other hand,   if  $N \geq e^{(1+\varepsilon)n\log 2}=2^{(1+\varepsilon)n}$, the left-hand side of \eqref{eq:cube-threshold} yields $$F_{n,N}(\mu_n) \geq 1-\exp\left(-2^{n\varepsilon}\right)> 1-\delta_n$$ for $n \geq n_2(\varepsilon)$, which means that \begin{equation}\label{eq:cube-rho2}      
    \varrho_2(\mu_n,\delta_n) \leq (1+\varepsilon)n\log 2. \end{equation}
    Thus, \eqref{eq:cube-rho1} and \eqref{eq:cube-rho2} complete the proof.
\end{proof}

  \begin{remark}
    If $\mu$ is the uniform probability measure on $\{0,1\}$, direct computation (see \eqref{eq:expectation-bernoulli} below) shows that $\mathbb{E}_{\mu}[\Lambda_\mu^\ast]=\log 2$ . So, the sharp threshold in Corollary~\ref{cor:cube-regimes} is again around $(\mathbb{E}_{\mu_n}[\Lambda_{\mu_n}^\ast])_{n \in \mathbb{N}}$.
\end{remark}

\subsection{Lattice points in $\ell_p$-ball}

This subsection is devoted to the proof of Theorem~\ref{thm:lattice-ball-threshold}.
Fix $r,p \geq 1$ and set $k=\lfloor r^p \rfloor$. For simplicity, we drop the parameters $r,p$ from the notation. Set now
\[
L_n:=\Z ^n\cap rB_p^n,\qquad M_n:=|L_n|.
\]
Let $\mu_n$ be the uniform probability measure on $L_n$ and $X_1,\dots,X_N$ be i.i.d with law $\mu_n$. Set  $K_N:=\conv\{X_1,\dots,X_N\}$.
We consider the normalized lattice-point functional
\[
F_{n,N}(\mu_n):=\E_{\mu_n^N}\Big[\frac{|K_N\cap\Z^n|}{|L_n|}\Big]
=\E_{\mu_n^N}\Big[\frac{|K_N\cap L_n|}{M_n}\Big].
\]

\medskip

Let $D_N:=|\{X_1,\dots,X_N\}|$ be the number of distinct sampled points.
Then we always have 
\begin{equation}\label{eq:ball-lower}
    D_N\le |K_N\cap L_n|.
\end{equation}
Set now $A_n:=\{x\in\{-1,0,1\}^n:\ \|x\|_0=k\}$. The cardinality of this set  is\begin{equation}\label{eq:An-cardinality}
|A_n|=2^k\binom{n}{k}.
\end{equation}
Observe that $A_n \subseteq L_n$, as for every $x \in A_n$ we have $\|x\|_p^p=k \leq r^p$.

Let us also consider the event $$G_N=(A_n \cap K_N)\setminus\{X_1,\ldots,X_N\}. $$ With this notation we have 
\begin{equation}\label{eq:ball-upper}
    |K_N\cap L_n|\leq D_N + |G_N|+|L_n\setminus A_n|.
\end{equation}

Combining \eqref{eq:ball-lower} and \eqref{eq:ball-upper}, dividing by $M_n$, taking expectations and using Lemma~~\ref{lem:coupon-collector}, we get

\begin{equation}\label{eq:ball-sandwich}
1-\Big(1-\frac{1}{M_n}\Big)^N
\ \le\
F_{n,N}(\mu_n)
\ \le\
1-\Big(1-\frac{1}{M_n}\Big)^N +C_n+P_{n,N},
\end{equation}
where $$C_n=\frac{|L_n \setminus A_n|}{M_n}, \qquad P_{n,N}=\frac{\mathbb{E}_{\mu_n^N}[|G_N|]}{M_n}.$$
For the inequality \eqref{eq:ball-sandwich} to be likely to yield a threshold, we must ensure that $C_n, P_{n,N} \to 0$, as $n \to \infty$ and for appropriate selection of $N$. To achieve this, we need to estimate the quantities involved. We begin with $M_n=|L_n|=|\mathbb{Z}^n \cap rB_p^n|$ and $C_n$.

\begin{lemma}\label{lem:Mn-Cn-asympt}
We have $$|L_n \setminus A_n|=O(n^{k-1})$$ and
\[
M_n = 2^k\binom{n}{k} + O(n^{k-1}).
\] Therefore, $$M_n \sim \frac{2^k}{k!}n^k, \qquad \log M_n \sim k \log n,$$ as $n \to \infty$, and $C_n=O(1/n)$ .
\end{lemma}

\begin{proof}

Write $$L_n= A_n \sqcup (L_n \setminus A_n).$$
We claim that $L_n\setminus A_n$ contains only vectors with at most $k-1$ non-zero coordinates.
Indeed, if $\|x\|_0 \geq k$ and some coordinate satisfies $|x_i|\ge 2$, then
\[
\|x\|_p^p \ge 2^p + (k-1)\cdot 1  \geq k+1 > r^p,
\]
which means that $x \notin L_n$. Thus, any $x\in L_n\setminus A_n$ has $\|x\|_0\le k-1$.

For each $1 \leq j\le k-1$, the number of integer vectors with $\|x\|_0=j$ and $\|x\|_p\le r$ is at most
$\binom{n}{j}(2\lfloor r\rfloor)^j$. Summing over $1 \leq j\le k-1$ gives
$|L_n\setminus A_n|=O(n^{k-1})$. Indeed, for each fixed $j$ we have $\binom{n}{j}\le n^j/j!$, hence
\[
\sum_{j=0}^{k-1}\binom{n}{j}(2\lfloor r\rfloor)^j
\le \sum_{j=0}^{k-1}\frac{(2\lfloor r\rfloor)^j}{j!}\,n^j
\le \Big(\sum_{j=0}^{k-1}\frac{(2\lfloor r\rfloor)^j}{j!}\Big)\,n^{k-1} \leq e^{2\lfloor r \rfloor} n^{k-1}.
\]
Since $r\geq1$ is fixed, $|L_n\setminus A_n|=O(n^{k-1})$.
 Combining with \eqref{eq:An-cardinality} concludes the proof.
\end{proof}

We must now estimate $P_{n,N}$. We will need the following lemma.

\begin{lemma}\label{lem:l1-supporting}
Let $x \in A_n$.
\begin{enumerate}
\item[{\rm (i)}] For  every $y \in L_n$ we have that $\langle x,y \rangle \leq k$.

\item[{\rm (ii)}] There exists a set $F_x \subseteq L_n$  such that $$|F_x|=\binom{2k-1}{k} \coloneq R_k$$ and $\langle x, y\rangle = k$ for some $y \in L_n$  if and only if $y \in F_x$. Consequently, if $x \in \conv(B)$ for some $B \subseteq L_n$, then $B \cap F_x \neq \emptyset$; moreover, if $x \notin B $, then $|B \cap F_x| \geq 2$.
\end{enumerate}
\end{lemma}

\begin{proof}
    Let $x \in A_n$. Since $A_n$ and $L_n$ are invariant under sign flips and permutations of the coordinates, we may assume that $x=(1,\ldots,1,0,\ldots,0)$. We set $$F_x=\{ y \in L_n:y_1, \ldots,y_k \geq 0, \ y_{k+1}=\cdots = y_n=0, \ \|y\|_1=k\}.$$ 
    We now compute $|F_x|$. Since every $y\in F_x$ satisfies
$y_{k+1}=\cdots=y_n=0$, $y_1,\ldots,y_k\ge 0$ and $\|y\|_1=k$, the map
\[
y\ \longmapsto\ (y_1,\ldots,y_k)
\]
is a bijection between $F_x$ and the set
\[
\Big\{(a_1,\ldots,a_k)\in \mathbb{Z}_{\ge 0}^k:\ a_1+\cdots+a_k=k\Big\}.
\]
Hence $|F_x|$ equals the number of $k$-tuples of non-negative integers
$(a_1,\ldots,a_k)$ summing to $k$.

To count these $k$-tuples, we use the standard ``stars and bars'' argument.
Represent $k$ identical objects (stars) to be distributed into $k$ bins, where
$a_i$ is the number of stars placed in bin $i$. Such a distribution can be encoded
by a word of length $2k-1$ consisting of $k$ stars and $k-1$ bars: reading from left
to right, $a_1$ is the number of stars before the first bar, $a_2$ the number of stars
between the first and second bar, and so on, with $a_k$ the number of stars after the last bar.
This correspondence is bijective, so the number of such $k$-tuples is the number of ways to choose
the positions of the $k$ stars (or equivalently the $k-1$ bars) among the $2k-1$ slots:
\[
|F_x|
=\binom{2k-1}{k}
=\binom{2k-1}{k-1}.
\]
Now let $y \in L_n$. We have \begin{equation}\label{eq:conditions}  
    \langle x,y \rangle = \sum_{i=1}^{k}y_i \leq \sum_{i=1}^{k}|y_i| \leq \|y\|_1 \leq k,
    \end{equation} since $\|y\|_1$ is an integer and $$\|y\|_1=\sum_{i=1}^{n} |y_i| \leq \sum_{i=1}^{n}|y_i|^p=\|y\|_p^p \leq r^p<k+1,$$ as $y_i$ are integers and $p \geq 1$. So, if $\langle x,y \rangle=k$, then all the inequalities in \eqref{eq:conditions} must be equalities. Hence, $$y_1,\ldots,y_k \geq 0, \qquad y_{k+1}=\cdots=y_n=0, \qquad \|y\|_1=k,$$ which means that $y \in F_x$. The converse is immediate.

     Let now $B \subseteq L_n$ with $x \in \conv(B)$. Then $x=\sum_{y\in B}\lambda_y y$ with $\lambda_y\ge0$, $\sum_{y \in B}\lambda_y=1$.
Taking inner products with $x$ gives
\[
k=\langle x,x\rangle=\sum_{y\in B}\lambda_y\langle x,y\rangle,
\]
and since each $\langle x,y\rangle\le k$, we must have $\langle x,y\rangle=k$ for every $y$ with $\lambda_y>0$.
Hence at least one such $y$ lies in $F_{x}$.
If $x\notin B$, then the convex combination representing $x$ uses at least two distinct points, so $|B\cap F_{x}|\ge 2$.
\end{proof}

Now we can estimate $P_{n,N}$:

\begin{lemma}\label{lem:PnN-estimate}
For every $x \in A_n$ we have $$\mu_n^N (x \in K_N \setminus \{X_1,\ldots,X_N\}) \leq \binom{N}{2} \left(\frac{R_k}{M_n}\right)^2,$$ therefore $$P_{n,N} = O\left( \frac{N^2}{M_n^2}\right).$$ 
\end{lemma}

\begin{proof}
    If $x \in K_N \setminus \{X_1,\ldots,X_N\}$, then by Lemma~\ref{lem:l1-supporting} we must have $|\{X_1,\ldots,X_N\} \cap F_x| \geq 2$. Since $X_i$ are independent random vectors, uniformly distributed on $L_n$, if we set $Y \sim {\rm Binomial}(N,|F_x|/M_n)$, we get $$\mu_n^N(|\{X_1,\ldots,X_N\} \cap F_x| \geq 2) \leq \mathbb{P}(Y \geq 2),$$ which yields \begin{equation}\label{eq:binomial-bound}
    \mu_n^N(x \in K_N \setminus \{X_1,\ldots,X_N\}) \leq \binom{N}{2} \left(\frac{R_k}{M_n}\right)^2.\end{equation} To conclude, observe that $$|G_N|= \sum_{x \in A_n} \mathds{1}_{\{x \in K_N \setminus \{X_1,\ldots,X_N\} \}}.$$ Taking expectations, dividing by $M_n$ and using \eqref{eq:binomial-bound} and $|A_n| \leq M_n$, we conclude the proof, since $R_k$ depends only on $r,p$.
\end{proof}

We are now ready to prove Theorem~\ref{thm:lattice-ball-threshold}.

\begin{proof}[Proof of Theorem~\ref{thm:lattice-ball-threshold}]
    We set $T_n=\log M_n$ and recall that, by Lemma~\ref{lem:Mn-Cn-asympt}, \ $T_n \sim k \log n$, as $n \to \infty$, where $k=\lfloor r^p \rfloor$. Let $\varepsilon \in (0,1)$ and set $\delta_n=1/T_n$.

     If $N \leq e^{(1-\varepsilon)T_n}=M_n^{1-\varepsilon}$, the right-hand side inequality of \eqref{eq:ball-sandwich}, combined with Lemma~\ref{lem:Mn-Cn-asympt} and Lemma~\ref{lem:PnN-estimate}, gives \begin{align*}
        F_{n,N}(\mu_n) &\leq 1-\left(1-\frac{1}{M_n}\right)^N+O(1/n) + O\left( \frac{N^2}{M_n^2}\right)  
        \\ &\leq \frac{N}{M_n}+ O\left( \frac{N^2}{M_n^2}\right) +O(1/n) 
        \\ &\leq \frac{1}{M_n^\varepsilon}+ O\left(\frac{1}{M_n^{2\varepsilon}}\right)+O(1/n) \leq \delta_n,
    \end{align*}
    for $n \geq n_1(\varepsilon)$, where we also used the inequality  $(1-x)^N \geq 1-Nx, \ x \in [0,1]$. Hence, \begin{equation}\label{eq:lattice-ball-rho1}
    \varrho_1(\mu_n,\delta_n) \geq (1-\varepsilon)T_n.    
    \end{equation}
    If $N \geq e^{(1+\varepsilon)T_n}=M_n^{1+\varepsilon}$, the left-hand side inequality of \eqref{eq:ball-sandwich} yields \begin{align*}       
    F_{n,N}(\mu_n) &\geq 1-\left(1-\frac{1}{M_n}\right)^N \geq 1-\exp\left(-\frac{N}{M_n}\right)
    \\ &\geq 1-\exp\left(-M_n^\varepsilon\right) \geq 1-\delta_n,
    \end{align*}
    for $n \geq n_2(\varepsilon)$, where we also used the inequality $(1-x) \leq e^{-x}$. Therefore, \begin{equation}\label{eq:lattice-ball-rho2}
        \varrho_2(\mu_n,\delta_n) \leq (1+\varepsilon)T_n.
    \end{equation}
    Combining \eqref{eq:lattice-ball-rho1} and \eqref{eq:lattice-ball-rho2} we obtain the sharp threshold.   
\end{proof}

\begin{remark}\label{rem:lattice-ball}
    For $r=1$, we have $L_n=\mathbb{Z}^n \cap B_p^n=\{0,\pm e_1,\ldots,\pm e_n\}$ for every $p \geq 1$. So, $M_n=2n+1$ and with the previous notation $T_n=\log (2n+1).$ Direct computation shows that $$\Lambda_{\mu_n}(\xi)= \log \left(1+\sum_{i=1}^n (e^{\xi_i}+e^{-\xi_i})\right)-\log(2n+1)$$ for every $\xi \in \mathbb{R}^n$. Moreover, we have $\Lambda_{\mu_n}^\ast(0)=0$ and $\Lambda_{\mu_n}^\ast(\pm e_i)= \log(2n+1)$ for every $1 \leq i \leq n$. Therefore, $$\mathbb{E}_{\mu_n}[\Lambda_{\mu_n}^\ast]=\frac{2n}{2n+1}\log(2n+1),$$ which means that $\mathbb{E}_{\mu_n}[\Lambda_{\mu_n}^\ast]\sim T_n$, as $n \to \infty$.
\end{remark}

\section{Counterexamples}

\subsection{The quantity $\beta(\mu)$ and counterexamples for product thresholds}
The \emph{relative-variance parameter}
\[
\beta(\mu)\ :=\ \frac{\mathrm{Var}_{\mu}(\Lambda^{\ast}_\mu)}{\big(\E_{\mu}[\Lambda^{\ast}_\mu\big])^2}
\]
was introduced in \cite{BGP} as a quantitative measure of how tightly the random variable
$\Lambda^{\ast}_\mu$ concentrates around its mean. In the \emph{continuous} log-concave setting (log-concave
probability measures on $\R^n$ with a \emph{continuous} density), this parameter is directly linked to the existence
of a threshold phenomenon for the expected $\mu$--mass of random polytopes, and hence for quantities
governed by $\E_\mu[\Lambda^{\ast}_\mu]$; see \cite{BGP} and the further developments in
\cite{BC}. In particular, a natural route to obtain threshold statements for \emph{all} continuous
log-concave measures is to prove that
\[
\beta_n\ :=\ \sup\{\beta(\mu):\ \mu \text{ is log-concave on }\R^n \text{ with continuous density}\}
\]
satisfies $\beta_n\to 0$ as $n\to\infty$. The corresponding conjecture has been formulated for the above class
and, to the best of our knowledge, no counterexample is currently known.

\medskip

Moreover, in \cite{BC} threshold behaviour was established in the continuous setting for \emph{product} measures on $\R^n$
whose one-dimensional marginals are log-concave but not necessarily identical (i.e., independent coordinates with
different one-dimensional log-concave laws). By contrast, in the \emph{discrete} log-concave category both phenomena
may fail: first, the quantity $\beta(\mu)$ can be unbounded even in fixed dimension, and second, the above
``product-stability'' of threshold may break down (an explicit discrete counterexample is discussed later).
We begin by illustrating the first obstruction, namely that $$\tilde\beta_n=\sup\{ \beta(\mu): \mu \text{ is discrete log-concave on } \mathbb{R}^n\} $$ can be infinite.
\medskip

We first notice that in the continuous log-concave case we always have that $\beta_1 < +\infty$ (see \cite{BC}). However, in the discrete log-concave setting, this is not true. More precisely, we can show that:

\begin{proposition}
    For every $n \in \mathbb{N}$ we have $\tilde{\beta}_n=+\infty$.
\end{proposition}

\begin{proof}
    If $p \in (0,1)$, let $\nu^p $ be the probability measure on $\mathbb{R}$ with $\nu^p(\{0\})=1-p$ and $\nu^p(\{1\})=p $. We also set $\nu^p_n=\nu^p\otimes\cdots\otimes\nu^p$. Simple calculations show that $$\Lambda_{\nu^p}(\xi)=\log (1-p+pe^\xi), \ \xi \in \mathbb{R}$$
    and $$\Lambda^{\ast}_{\nu^p}(0)=-\log(1-p), \qquad \Lambda^{\ast}_{\nu^p}(1)=-\log p.$$
    Therefore, we have \begin{equation}\label{eq:expectation-bernoulli}
    \mathbb{E}_{\nu^p}[\Lambda^{\ast}_{\nu^p}]=-p \log p -(1-p) \log(1-p)
    \end{equation}
    and $${\rm Var}_{\nu^p}(\Lambda^{\ast}_{\nu^p})=p(1-p)\log^2\left( \frac{p}{1-p} \right).$$
    Notice that 
    \begin{equation}\label{eq:asymptotics}\mathbb{E}_{\nu^p}[\Lambda^{\ast}_{\nu^p}] \sim -p \log p, \qquad {\rm Var}_{\nu^p}(\Lambda^{\ast}_{\nu^p}) \sim p \log^2 p
    \end{equation} 
    as $p \to 0$. So, there exist $c,C>0$ such that for every $p \in (0,1/2]:$
        $$\mathbb{E}_{\nu^p}[\Lambda^{\ast}_{\nu^p}] \leq -Cp \log p, \qquad {\rm Var}_{\nu^p}(\Lambda^{\ast}_{\nu^p}) \geq cp \log^2 p.$$ For every $x \in \{0,1\}^n$ we have $\Lambda^{\ast}_{\nu^p_n}(x)= \sum\limits_{k=1}^{n} \Lambda^{\ast}_{\nu^p}(x_k)$ which means that $$ \qquad \mathbb{E}_{\nu^p_n}[\Lambda^{\ast}_{\nu^p_n}]=\sum\limits_{k=1}^{n} \mathbb{E}_{\nu^p}[\Lambda^{\ast}_{\nu^p}], \qquad {\rm Var}_{\nu^p_n}(\Lambda^{\ast}_{\nu^p_n})=\sum\limits_{k=1}^{n} {\rm Var}_{\nu^p}(\Lambda^{\ast}_{\nu^p}).$$ 
    Thus, for every $p \in (0, 1/2]$, $$\beta(\nu^p_n)=\frac{{\rm Var}_{\nu^p_n}(\Lambda^{\ast}_{\nu^p_n})}{(\mathbb{E}_{\nu^p_n}[\Lambda^{\ast}_{\nu^p_n}])^2} \geq \frac{c}{C^2}\frac{1}{np},$$ which gives that $$\lim\limits_{p \to 0} \beta(\nu_n^p)=+\infty.$$
    \end{proof}

We now move to the case of the product of $n$ not necessarily identical measures. In the continuous log-concave setting we know that if $(\lambda_n)_{n \in \mathbb{N}}$ is a sequence of log-concave probability measures on $\mathbb{R}$ and $\mu_n= \lambda_1 \otimes\cdots\otimes \lambda_n$ for every $n \in \mathbb{N}$, then $$\beta(\mu_n) \to 0,$$ as $n \to \infty$ and $(\mu_n)_{n \in \mathbb{N}}$ exhibits a sharp threshold. We will show that in the discrete case this is no longer true.

    \begin{theorem}\label{th:discr-prod}
        There exists a sequence $(\lambda_n)_{n \in \mathbb{N}}$ of discrete log-concave probability measures on $\mathbb{R}$ such that if $\mu_n= \lambda_1 \otimes\cdots\otimes \lambda_n$ for every $n \in \mathbb{N}$, then $$\beta(\mu_n) \to +\infty,$$ as $n \to \infty$ and  $(\mu_n)_{n \in \mathbb{N}}$ does not exhibit a sharp threshold.
    \end{theorem}

    \begin{proof}
        For every $k \in \mathbb{N}$ we set $$p_k=\frac{1}{k \log^3(k+3)}$$ and  $\lambda_k=\nu^{p_k}$, following the previous notation. 
        
        \begin{lemma}
            We have $\beta(\mu_n) \to \infty$, as $n \to \infty$.
        \end{lemma}

\begin{proof}

        For every $x \in \{0,1\}^n$ we have $\Lambda^{\ast}_{\mu_n}(x)= \sum\limits_{k=1}^{n} \Lambda^{\ast}_{\lambda_k}(x_k)$, which gives that $$ \mathbb{E}_{\mu_n}[\Lambda^{\ast}_{\mu_n}]=\sum\limits_{k=1}^{n} \mathbb{E}_{\lambda_k}[\Lambda^{\ast}_{\lambda_k}], \qquad {\rm Var}_{\mu_n}(\Lambda^{\ast}_{\mu_n})=\sum\limits_{k=1}^{n} {\rm Var}_{\lambda_k}(\Lambda^{\ast}_{\lambda_k}).$$ By \eqref{eq:asymptotics}, there exist $c,C>0$ such that, for every $p \in (0,1/2]$,
        $$\mathbb{E}_{\nu^p}[\Lambda^{\ast}_{\nu^p}] \leq -Cp \log p, \qquad {\rm Var}_{\nu^p}(\Lambda^{\ast}_{\nu^p}) \geq cp \log^2 p. $$ So, now we have $$\mathbb{E}_{\mu_n}[\Lambda^{\ast}_{\mu_n}] \leq C \sum\limits_{k=1}^{n} p_k \log\left(\frac{1}{p_k}\right) \leq C \sum\limits_{k=1}^{\infty} \frac{\log\left(k \log^3(k+3)\right)}{k \log^3(k+3)} \leq \tilde{C}\sum\limits_{k=1}^{\infty}\frac{1}{k \log^2(k+3) }< +\infty,$$ which means that $$\mathbb{E}_{\mu_n}[\Lambda^{\ast}_{\mu_n}] \to A \in (0,+\infty),$$ as $n \to \infty$. However, $${\rm Var}_{\mu_n}(\Lambda^{\ast}_{\mu_n}) \geq c\sum\limits_{k=1}^{n}p_k \log^2\left(\frac{1}{p_k}\right)=c\sum\limits_{k=1}^{n} \frac{\log^2\left(k \log^3(k+3)\right)}{k \log^3(k+3)}\geq \tilde{c} \sum\limits_{k=1}^{n} \frac{1}{k \log(k+3)} \to +\infty,$$ as $n \to \infty$. This means that $$\beta(\mu_n)=\frac{{\rm Var}_{\mu_n}(\Lambda^{\ast}_{\mu_n})}{(\mathbb{E}_{\mu_n}(\Lambda^{\ast}_{\mu_n}))^2} \to +\infty,$$ as $n \to \infty$.   \end{proof}

We now proceed to show that  $(\mu_n)_{n \in \mathbb{N}}$  cannot exhibit any  threshold.

        \begin{lemma}\label{lem:counter-lower}
        There exists some $\delta_0 \in (0, 1/2)$ such that for every $n,N \in \mathbb{N}$ we have $$\mathbb{E}_{\mu_n^N}[\mu_n(K_N)] \geq \delta_0.$$ 
            
        \end{lemma}
        
\begin{proof}
Let $n,N \in \mathbb{N}$. It is clear that $K_1 \subseteq K_N$ and we have 
$$ \mathbb{E}_{\mu_n^N}[\mu_n(K_1)] = \sum\limits_{x \in \{0,1\}^n } \mu^2_n(\{x\}) \geq \mu^2_n(\{(0,\ldots,0)\})$$ and $$\mu_n(\{(0,\ldots,0)\})=\prod\limits_{k=1}^{n} \lambda_k(\{0\})= \prod\limits_{k=1}^{n}(1-p_k) \geq \prod_{k=1}^{\infty}(1-p_k)>0,$$ since $\sum_{k=1}^{\infty}p_k < +\infty$. So, if we set $$\delta_0=\left(\prod\limits_{k=1}^{\infty}(1-p_k)\right)^2 \in (0,1/2)$$ we get $\mathbb{E}_{\mu_n^N}[\mu_n(K_N)] \geq \delta_0$. \end{proof}

    \begin{lemma}\label{lem:counter-upper}Let  $n \in \mathbb{N}$. For every $\delta>0$ there exists $N_0(n,\delta) \in \mathbb{N}$, such that for every $N \geq N_0(n,\delta)$ we have $$\mathbb{E}_{\mu_n^N}[\mu_n(K_N)] \geq 1-\delta.$$ In addition, if $n \geq 2$ and $0<\delta<\delta_0$, then we have $ \varrho_2(\mu_n,\delta) \geq \log 2 $.
        
    \end{lemma}

    \begin{proof}
        Let $\delta>0$ and $(X_k)_{k \in \mathbb{N}}$ be a sequence of independent random vectors, distributed according to $\mu_n$. By Lemma~\ref{lem:no-new-verts} $$\{ \{X_1,\ldots,X_N\}=\{0,1\}^n\}=\{\mu_n(K_N)=1\}$$ and $$\lim\limits_{N \to \infty} \mu_n^N\left(\{X_1,\ldots,X_N\}=\{0,1\}^n\right)=1.$$ So, there exists $N_0(n,\delta) \in \mathbb{N}$, such that, for every $N \geq N_0(n,\delta)$, $$\mu_n^N\left(\{X_1,\ldots,X_N\}=\{0,1\}^n\right) >1-\delta.$$ Therefore, by Markov's inequality $$\mathbb{E}_{\mu_n^N}[\mu_n(K_N)] >1-\delta$$ for every $N \geq N_0(n,\delta)$, which implies that $\varrho_2(\mu_n,\delta)$ is well-defined.

         Now, if $n \geq 2$ and $0<\delta<\delta_0$, we have
$$\mathbb{E}[\mu_n(K_1)] \leq \mu_n(\{0,\ldots,0\})= \prod\limits_{k=1}^n (1-p_k) \leq1/2.$$ So, if $$\mathbb{E}_{\mu_n^N}[\mu_n(K_N)]>1-\delta$$ for some $N \in \mathbb{N}$, we must have $N \geq 2$, which yields $\varrho_2(\mu_n,\delta) \geq \log 2$.
\end{proof}

    Let now $n \geq 2$ and $0<\delta<\delta_0$. Then, by Lemma~\ref{lem:counter-lower},  $\varrho_1(\mu_n,\delta)$ is not well-defined, since $\mathbb{E}_{\mu_n^N}[\mu_n(K_N)]$ never gets arbitrarily small. However, by Lemma~\ref{lem:counter-upper} we have that $\varrho_2(\mu_n,\delta)$ is well-defined, and in fact $ \varrho_2(\mu_n,\delta)\geq \log 2$. Thus, $(\mu_n)_{n \in \mathbb{N}}$ cannot exhibit any threshold, according to our definition.    \end{proof}

\subsection{The expectations of $q_\mu$ and $\Lambda_{\mu}^\ast$}

Answering a question that was asked in \cite{mathoverflow}, Brazitikos, Giannopoulos and Pafis showed in \cite{BGP2} the following result concerning the expectation of Tukey's half-space depth of a (continuous) log-concave probability measure:

\begin{theorem}\label{thm:Tukey-expectation}
    There exist absolute constants $c_1,c_2>0$ such that for every log-concave probability measure $\mu$ on $\mathbb{R}^n$ we have $$e^{-c_1n} \leq \mathbb{E}_{\mu}[q_\mu] \leq e^{-c_2n}.$$
\end{theorem}

The latter statement takes also into account the resolution of the isotropic constant conjecture, initially by Klartag and Lehec \cite{KlartagLehec}, and shortly thereafter by Bizeul \cite{Biz}. So, one may wonder whether the same bounds hold in the discrete log-concave setting. The answer for the upper bound is negative.

\begin{proposition}
For every $n \in \mathbb{N}$ we have $$\sup_X \, \mathbb{E}_X[q_X]=1,$$ where the supremum is over all random vectors $X$ in $\mathbb{Z}^n$ with independent and discrete log-concave coordinates.
\end{proposition}

\begin{proof}
    Let $\varepsilon \in (0,1)$ and consider independent random variables $X_1,\ldots,X_n$  with $X_i \sim \text{Bernoulli}(p_i)$, where $p_i=\frac{\varepsilon}{2^{i+1}}$ for every $1 \leq i \leq n$. If we set $X=(X_1,\ldots,X_n)$, then $q_X(x)=\mathbb{P}(X=x)$ for every $x \in \{0,1\}^n$. Thus, \begin{align*} \mathbb{E}_X[q_X] &=\sum\limits_{x \in \{0,1\}^n} \mathbb{P}(X=x)^2=\sum\limits_{x \in\{0,1\}^n} \prod\limits_{i=1}^{n} \mathbb{P}(X_i=x_i)^2
    \\ &=\prod\limits_{i=1}^{n}\ \sum\limits_{x \in \{0,1\}} \mathbb{P}(X_i=x_i)^2 = \prod\limits_{i=1}^n (1-2p_i(1-p_i)) 
    \\ &\geq 1-2\sum\limits_{i=1}^n p_i(1-p_i) \geq 1-\varepsilon.
    \end{align*}
\end{proof}

\begin{remark}
    In fact, in the proof of the upper bound of Theorem~\ref{thm:Tukey-expectation} the authors showed that if $\mu$ is a log-concave probability measure on $\mathbb{R}^n$, then $$\E_{\mu}[e^{-\Lambda_{\mu}^\ast}] \leq e^{-c_1n},$$ where $c_1>0$ is an absolute constant. Applying Jensen's inequality we get $$\mathbb{E}_{\mu}[\Lambda_{\mu}^\ast] \geq c_1 n.$$
They also demonstrated in \cite[Corollary 1.5]{BGP2} that there exists an absolute constant $c_2>0$ such that 
 $$\sup_{\mu }\Big(\sup\Big\{{\mathbb E}_{\mu^N}[\mu (K_N)]:N\leq e^{c_1n}\Big\}\Big)\longrightarrow 0$$
 as $n \to \infty$, where the first supremum is over all log-concave probability measures $\mu$ on $\mathbb{R}^n$. Theorem~\ref{thm:lattice-ball-threshold} combined with Remark~\ref{rem:lattice-ball}, and especially \eqref{eq:expectation-bernoulli} along with Lemma~\ref{lem:counter-lower}, show that no analogue of either result can hold in the discrete log-concave setting.
\end{remark}

Let now $\mu$ be a probability measure on $\mathbb{R}$. Considering the inequality $ e^{\Lambda_{\mu}^\ast } \leq 1/q_{\mu}$ and \cite[Proposition 3.2 b)]{BC}, we now know that $\mathbb{E}_{\mu}[e^{\alpha\Lambda_{\mu}^\ast}]<+\infty$ for every $\alpha \in (0,1)$. In particular, $\Lambda_{\mu}^{\ast}$ has all its moments finite. In the following, we present examples of probability measures on $\mathbb{R}^n$ with specific properties that nevertheless do not imply finiteness even of $\mathbb{E}_{\mu}[\Lambda_{\mu}^\ast]$. 

Before the result of Brazitikos and Chasapis that we mentioned, it was known that if a probability measure $\mu$ on $\mathbb{R}$  has finite first moment, then $\mathbb{E}_{\mu}[e^{\alpha\Lambda_{\mu}^\ast}]<+\infty$ for every $\alpha \in (0,1)$ (see \cite[Exercise 1.2.9]{DeuStr}).  We first show that this property does not extend to $\mathbb{R}^n, \, n \geq 2$. In fact, we show something stronger. That even compact support is not enough to ensure that the first moment of $\Lambda_{\mu}^\ast$ is finite.

\begin{proposition}
    Let $n \geq 2$. There exists a full-dimensional, compactly supported probability measure $\mu_n$ on $\mathbb{R}^n$ with $\mathbb{E}_{\mu_n}[\Lambda_{\mu_n}^\ast]=+\infty.$
\end{proposition}

\begin{proof}
Let $n\geq 2$. We set $$x_2=0, \qquad x_k=(\cos (1/k), \sin (1/k),
\ 0\ , \ldots,\underbrace{1}_{k-\text{index} \ } ,\ 0\ ,\ldots,0), \qquad 3 \leq k\leq n $$ and $$x_k=(\cos(1/k),\sin(1/k), \ 0\ ,\ldots, \ 0), \qquad k \geq n+1 .$$ For every $k \geq 2$ we also set  $p_k=c/k \log^2 k$, where $c>0$ is a normalizing constant such that $\sum_{k=2}^{\infty}p_k=1$. We define the probability measure $$\mu_n= \sum\limits_{k=2}^{\infty} p_k \delta_{x_k},$$ where $\delta_{x_k}$ is the Dirac measure at $x_k$. Note that $\mu_n$ is full-dimensional. It is also clear that  since $\mu_n$ is compactly supported, then $\Lambda_{\mu_n}(\xi) < +\infty$ for every $\xi \in \mathbb{R}^n$.

 Let now $k \geq 3$ and $t>0$. Setting $\xi_k=(\cos(1/k),\sin(1/k), \ 0\ ,\ldots, \ 0)$, we have $$e^{\Lambda_{\mu_n}(t\xi_k)} = \sum\limits_{m=2}^{\infty}p_m e^{t\langle \xi_k, x_m \rangle} = e^{t \langle \xi_k, x_k\rangle} \left(p_k+ \sum\limits_{\substack{m=2 \\ m \neq k}}^{\infty}p_m e^{t \langle \xi_k,x_m-x_k \rangle} \right).$$
Now, if $m \geq 2$ and $m \neq k$ then $$\langle \xi_k, x_m-x_k \rangle \leq \cos \left(\frac{1}{k}-\frac{1}{m}\right)-1 \leq \cos\left( \frac{1}{k(k+1)} \right)-1=B_k<0.$$
Therefore, $$e^{\Lambda_{\mu_n}(t\xi_k)} \leq e^{t} \left(p_k +e^{tB_k}\sum\limits_{\substack{m=2 \\ m \neq k}}^{\infty}p_m \right) \leq e^t \left(p_k +e^{tB_k} \right).$$
Thus, $$\Lambda_{\mu_n}^{\ast}(x_k) \geq \langle t\xi_k, x_k\rangle -\Lambda_{\mu_n}(t\xi_k) \geq t-t-\log(p_k+e^{tB_k})=-\log(p_k+e^{tB_k}).$$
Letting  $t \to +\infty$, we get that $\Lambda_{\mu_n}^{\ast}(x_k) \geq \log (1/p_k)$ and consequently, since $\Lambda_{\mu_n}^{\ast}(x_2)\geq 0$, $$\mathbb{E}_{\mu_n}[\Lambda_{\mu_n}^{\ast}] \geq \sum\limits_{k=3}^{\infty} \Lambda_{\mu_n}^{\ast}(x_k)p_k \geq \sum\limits_{k=3}^{\infty} p_k \log \left( \frac{1}{p_k} \right)= +\infty. $$
\end{proof}

\begin{remark}
    One can also show that the sequence $(\mu_n)_{n \in \mathbb{N}}$ does not exhibit a threshold by slightly modifying Lemma~\ref{lem:counter-lower} and Lemma~\ref{lem:counter-upper}. 
\end{remark}

Let now $\mu$ be an absolutely continuous probability measure on $\mathbb{R}$. In \cite[Remark 3.3]{BC}, one can find a simpler proof of the fact that $\mathbb{E}_{\mu}[e^{\alpha\Lambda_{\mu}^\ast}]<+\infty$ for every $\alpha \in (0,1)$. One might ask whether this property is enough to ensure that $\mathbb{E}_{\mu}[\Lambda_{\mu}^{\ast}]$ is finite. However, we show that in $\mathbb{R}^n, \, n \geq 2 , $ this condition is not sufficient again.

\begin{proposition}
    Let $n \geq 2$. There exists a full-dimensional, absolutely continuous probability measure $\mu_n$ on $\mathbb{R}^n$ with $\mathbb{E}_{\mu_n}[\Lambda_{\mu_n}^\ast]=+\infty.$
\end{proposition}

\begin{proof}
Let $n \geq 2$. For every $k \geq 2$ we set $x_k=(k,k^2,0,\ldots,0) \in \mathbb{R}^n$ and $p_k=c/k\log^2k$, where $c>0$ is a normalizing constant such that $\sum_{k=2}^{\infty}p_k=1$. We also choose $r_k=\frac{2}{9k^2}$. Notice that $(B(x_k,r_k))_{k \geq 2}$ are disjoint sets.  We then define $\mu_n$ to be the probability measure on $\mathbb{R}^n$ with density $$f_n(x)=\sum\limits_{k=2}^{\infty} \frac{p_k}{\vol_n(B(x_k,r_k))} \cdot \mathds{1}_{B(x_k,r_k)}(x), \ x\in \mathbb{R}^n.$$
For every $k \geq 2$ we also set $\xi_k=(2bk\log k,-b\log k,0, \ldots, 0) \in \mathbb{R}^n$, where $b>0$ is a constant that will be determined later. Then for every $m,k \geq 2$ we have
$$\langle x_m,\xi_k \rangle=b\log k \ (k^2-(k-m)^2),$$ which gives that for every $m,k\geq 2$ with $m \neq k$, we have: \begin{equation}\label{eq:equation1}
\langle x_k,\xi_k \rangle=b k^2 \log k \, \qquad \langle \xi_k, x_m-x_k \rangle =-b \log k \ (k-m)^2.
\end{equation}
Now, let $k \geq 2 $. We write
\begin{equation}\label{eq:equation2}
\begin{aligned} e^{\Lambda_{\mu_n}(\xi_k)} &= \sum\limits_{m=2}^{\infty} \frac{p_m}{\vol_n(B(x_m,r_m))}\int\limits_{B(x_m,r_m)} e^{ \langle \xi_k,y\rangle } \, dy \\
&= e^{ \langle \xi_k,x_k \rangle } \Bigg(\frac{p_k}{\vol_n(B(x_k,r_k))} \int\limits_{B(x_k,r_k)} e^{ \langle \xi_k,y- x_k\rangle } \, dy \\ &\hspace{1cm}+\sum\limits_{\substack{m=2 \\ m \neq k}}^{\infty} \frac{p_m}{\vol_n(B(x_m,r_m))} \int\limits_{B(x_m,r_m)} e^{ \langle \xi_k,y- x_m\rangle + \langle \xi_k,x_m- x_k\rangle} \, dy \Bigg).
\end{aligned}
\end{equation}
Now for every $m \geq 2$ and $y \in B(x_m,r_m) $ we have \begin{equation}\label{eq:equation3} 
\langle \xi_k,y-x_m \rangle \leq \|\xi_k\|_2 \ r_m \leq 3b r_m k\log k.
\end{equation}
So, by \eqref{eq:equation1}, \eqref{eq:equation2} and \eqref{eq:equation3} we obtain:
\begin{equation}\label{eq:equation4}
\begin{aligned}e^{\Lambda_{\mu_n}(\xi_k)} &\leq e^{ \langle \xi_k, x_k \rangle} \left( p_k e^{3bk r_k \log k} + \sum\limits_{\substack{m=2 \\ m \neq k}}^{\infty} p_m e^{3br_mk \log k -b \log k (k-m)^2} \right)
\\ 
&= e^{ \langle \xi_k, x_k \rangle} p_k e^{3bk r_k \log k} \left(1+ \sum\limits_{\substack{m=2 \\ m \neq k}}^{\infty} \frac{p_m}{p_k} e^{b \log k (3k(r_m-r_k)-(k-m)^2)}\right)
\end{aligned}
\end{equation}
Now it is easy to show that 
\begin{equation}\label{eq:equation6}
3k(r_m-r_k) \leq 3kr_m =\frac{2k}{3m^2} \leq \frac{(k-m)^2}{2}
\end{equation}
for every $m,k \geq 2$ with $m \neq k$. Moreover, for $m,k \geq 2$: \begin{equation}\label{eq:equation7}
\frac{p_m}{p_k} = \frac{k \log^2k}{m \log^2 m} \leq A k \log^2 k,
\end{equation}
where $A=\frac{1}{2 \log^22}$. Using \eqref{eq:equation6} and \eqref{eq:equation7}, we get:
\begin{equation}\label{eq:equation5}
\begin{aligned} \sum\limits_{\substack{m=2 \\ m \neq k}}^{\infty} \frac{p_m}{p_k} e^{b \log k (3k(r_m-r_k)-(k-m)^2)} &\leq A k \log^2 k \sum\limits_{\substack{m=2 \\ m \neq k}}^{\infty} e^{-b  \frac{(k-m)^2}{2} \log k} 
\\
&\leq 2Ak \log^2 k \sum\limits_{m=1}^{\infty} e^{-b  \frac{m^2}{2} \log k } = 2Ak \log^2 k \sum\limits_{m=1}^{\infty} \left( \frac{1}{k^{b/2}}\right)^{m^2} 
\\
&\leq 4A k^{1-b/2} \log^2 k ,
\end{aligned}
\end{equation}
if $b \geq 2$.  We choose $b=3$. Thus, by \eqref{eq:equation4} and \eqref{eq:equation5}, there exists $k_0 \in \mathbb{N}$ such that for every $k \geq k_0$: $$e^{\Lambda_{\mu_n}(\xi_k)} \leq 4e^{\langle x_k , \xi_k \rangle}p_k, \qquad \|\xi_k\|_2 \ r_k<1.$$
Let now  $k \geq k_0$ and $x \in B(x_k,r_k)$. We have
\begin{align*}
   \Lambda_{\mu_n}^{\ast}(x) &\geq \langle \xi_k,x \rangle - \Lambda_{\mu_n}(\xi_k) \geq \langle \xi_k,x-x_k\rangle -\log (4p_k)  \\
&\geq -\|\xi_k\|_2r_k-\log(4p_k) \geq -1-\log(4p_k).
\end{align*}
Therefore, \begin{align*}\mathbb{E}_{\mu_n}[\Lambda_{\mu_n}^{\ast}] & \geq C+\sum\limits_{k=k_0}^{\infty} \frac{p_k}{\vol_n(B(x_k,r_k))} \int\limits_{B(x_k,r_k)} \Lambda_{\mu_n}^{\ast}(x) \, dx
\\
&\geq C-1+\sum\limits_{k=k_0}^{\infty} p_k \log\left(\frac{1}{4p_k}\right)=+\infty.
\end{align*}
\end{proof}

\bigskip
\textbf{Acknowledgments.} We would like to thank Apostolos Giannopoulos for valuable comments and suggestions and Mihalis Kolountzakis for the example in Remark \ref{koloun}.

\noindent\textbf{Mathematics Subject Classification (2020).}
Primary: 60D05, 52A22, 60F10; Secondary: 52B05, 60G70, 39A12.

\noindent\textbf{Keywords.}
random polytopes; Tukey depth; Cram\'er transform; discrete log-concavity; threshold phenomena.

\end{document}